\documentclass[11pt,reqno]{amsart}
\usepackage{amssymb, bbm, color}
\usepackage[makeroom]{cancel}
\usepackage{ulem}
\usepackage{amsfonts}
\usepackage{mathrsfs}
\usepackage{amsmath}
\usepackage{graphicx}
\usepackage{hyperref}
\usepackage{float}
\usepackage{epstopdf}
\usepackage{color}
\usepackage{bm}
\usepackage{comment}
\usepackage{soul}
\usepackage{esint}
\usepackage[shortlabels]{enumitem}
\usepackage{eufrak}

\allowdisplaybreaks
\newtheorem{theorem}{Theorem}[section]
\newtheorem{proposition}[theorem]{Proposition}
\newtheorem{lemma}[theorem]{Lemma}
\newtheorem{corollary}[theorem]{Corollary}
\newtheorem{remark}[theorem]{Remark}

\newtheorem{definition}[theorem]{Definition}

\def\diam{\mathrm{diam}}

  \def\Ex{\mathfrak E}

\def\mcE{\mathcal{E}}
\def\mcF{\mathcal{F}}

\def\sB{\mathcal{B}}
\def\sE{\mathcal{E}}
\def\sF{\mathcal{F}}

\def\sN{\mathcal{N}}
\def\sX{\mathcal{X}}

\def\R{{\mathbbm R}}

\def\Z{{\mathbbm Z}}

\def\bP{{\mathbbm P}}

\def\bD{{\overline D}}
\def\1{{\mathbbm{1}}}
 
 \makeatletter
\@namedef{subjclassname@2020}{%
  \textup{2020} Mathematics Subject Classification}
\makeatother

\def\<{\langle}
\def\>{\rangle}

\numberwithin{equation}{section}

\begin{document}

 \title[Heat kernel for reflected jump diffusion on Ahlfors regular domains]{Heat kernel for reflected jump diffusion on Ahlfors regular domains}

\author{Shiping Cao}
\address{Department of Mathematics, The Chinese University of Hong Kong, Shatin, Hong Kong}
\email{spcao@math.cuhk.edu.hk}
\thanks{}

\author{Zhen-Qing Chen}
\address{Department of Mathematics, University of Washington, Seattle, WA 98195, USA}\email{zqchen@uw.edu}
\thanks{}


\date{}

  
\maketitle

\begin{abstract}
We study reflected jump diffusions on Ahlfors regular domains in general metric measure spaces. Under the condition that the Dirichlet form on the ambient space satisfies a capacity upper bound estimate, we construct an extension operator from the reflected Dirichlet space to the ambient Dirichlet space, with a scale-invariant local bound.  Second, we establish the mixed stable-like heat kernel estimates for the reflected jump diffusion, assuming that the process on the ambient space satisfies the same type of heat kernel estimates.
\end{abstract}

\maketitle
\section{Introduction} \label{S:1}
On inner uniform domains, where the ambient space supports a symmetric diffusion that satisfies Gaussian or sub-Gaussian heat kernel estimates, it is known that the corresponding reflected diffusion also satisfies 
Gaussian or sub-Gaussian heat kernel estimates  \cite{GL, Mathav, Ant}. 
A similar result holds for reflected diffusion processes with jumps in uniform domains in  Euclidean spaces \cite{CKKW}. 
See \cite{C,CF}   for general results on reflected Dirichlet form and reflected symmetric Markov processes.  
 
The aim of this paper is to study two-sided  heat kernel estimates  for reflected mixed stable-like  jump diffusions on a large class of open sets in locally compact separable metric spaces. In this paper, we say $(\sX, d, m)$ is a metric measure space if $(\sX, d)$ is a  locally compact separable metric space and $m$
is a Radon measure on $(\sX, d)$ having full support. 
The results of this paper have several important implications. For instance, they allow
 us to extend the boundary Harnack principle for censored processes \cite{BBC} to more general domains, using the recent 
results from \cite{CC1}. See also \cite{CC, KM} for recent studies on trace processes of reflected  
diffusions. 

Since connectivity of the underlying space is no longer required for jump processes, the inner uniform domain condition in the heat kernel study of reflected diffusions can be  drastically relaxed for reflected jump diffusions. It turns out that a natural and good condition for reflected jump diffusions is that the open set is
Ahlfors regular, a concept  introduced in Definition \ref{Ahlfors} below for metric measure spaces.
It is a natural extension of the classical Ahlfors $d$-regular set in the $d$-dimensional Euclidean space $\R^d$.
Recall that a Borel  measurable set $F\subset \R^n$ is said to be Ahlfors $d$-regular for some $d\in (0, n]$
if there is a  measure $m$ on $F$ so that $m(B(x,r))\asymp r^d$ for every $x\in \sX$ and $0<r<1$.
Here $B(x, r):= \{y\in F: |y-x|<r\}$, and the notation $f\asymp g$ means that is a constant $c\geq 1$ so that $c^{-1} g \leq f \leq c g$ on
the common domain of  $f$ and $g$. 
It is well known (cf. \cite{JW}) that the Radon measure $m$ in the above definition   is uniformly comparable to the $d$-dimensional Hausdorff measure ${\mathcal H}^d$ on $F$. 

Back in 2003, Chen and Kumagai \cite{CK} established  the $\alpha$-stable-like heat kernel estimate 
\[
p(t,x,y)\asymp t^{-d/\alpha}\wedge\frac{t}{|x-y|^{d+\alpha}}\quad \hbox{ for }0<t<1\hbox{ and }x,y\in \sX,
\]
for the Hunt process associated with the symmetric non-local Dirichlet form $(\sE, \sF)$
\begin{eqnarray}
\mcF &:=& \Big\{f\in L^2(\sX;m):\int_{\sX\times \sX}(f(x)-f(y))^2\frac{1}{|x-y|^{d+\alpha}}m(dx)m(dy)<\infty\Big\},  \label{e:1.1} \\
\mcE(f,g) &:=& \int_{\sX\times \sX}(f(x)-f(y))(g(x)-g(y))\frac{c( x,y)}{|x-y|^{d+\alpha}}m(dx)m(dy), 
\quad f, g\in \sF,   \label{e:1.2} 
\end{eqnarray}
on $L^2(\sX;m)$, where $\sX$ is a closed Ahlfors $d$-regular subset of $ \R^n$ with $d\in (0, n]$ with respect to the  measure $m$ on $\sX$ so that  $m(B(x, r))\leq c r^d$ for every $x\in \sX$ and $r>0$, 
and $c(x, y)$ is a symmetric measurable function on $\sX \times \sX$
that is bounded between two positive constants. The above heat kernel result has been  extended to more general mixed stable-like non-local Dirichlet forms in \cite{CK2} on metric measure spaces satisfying volume doubling (VD) and reverse volume doubling (RVD) conditions and a uniform volume comparability condition. We mention that the boundary of an  Ahlfors  $d$-regular open subset $D\subset \R^d$ can be highly irregular except that its boundary $\partial D$ has zero Lebesgue measure. Thus the results in \cite{CK, CK2} in particular
 give the mixed stable-like heat kernel estimates for the reflected stable-like processes on the closure of Ahlfors $d$-regular open subsets in $\R^d$.

We point out that when $\sX$ is an Ahlfors $d$-regular set in $\R^d$,
the parameter $\alpha$ has to be strictly less than 2,
as for $\alpha \geq 2$ the function space $\mcF$ defined by \eqref{e:1.1} consists of only constant functions. In more general metric measure spaces  including Sierpinski gaskets and Sierpinski carpets in $\R^n$, it may happen that $\alpha\geq 2$, in which case we need a cutoff Sobolev inequality for the stable-like, and more general mixed stable-like heat kernel estimates. 
 Recent developments on the necessary and sufficient conditions for such heat kernel estimates can be found in \cite{CKW, CKW2, GH, Ma}.  These works can be viewed as counterparts of the corresponding
 studies of diffusion processes and finite range  random walks, including \cite{AB,BB,BBK,BCM,BM,GHL}. In particular,  the following characterization of heat kernel estimates is established in  Chen-Kumagai-Wang \cite[Theorem 1.13]{CKW}. 
  Suppose that the metric measure space $(\sX, d, m)$ satisfies (VD) and (RVD) conditions, then the two-sided
  mixed stable-like heat kernel estimates hold for a Hunt process associated with a purely non-local symmetric regular 
  Dirichet form $(\sE, \sF)$ on $L^2(\sX; m)$ if and if only its jump kernel $J(x, y)$ satisfies a two-sided bounds ${\bf J}_\phi$
  and the Dirichlet form satisfies a cut-off Sobolev inequality condition  ${\rm CSJ}(\phi)$. See Section \ref{S:2} for 
  their definitions and a precise statement. Conditions ${\bf J}_\phi$ and  ${\rm CSJ}(\phi)$ are invariant under the comparable
  perturbation of jump kernels $J(x, y)$.  Stable characterization of  parabolic Harnack inequality for mixed stable-like symmetric Hunt processes
  have been established in \cite{CKW2}. 
  The (RVD) condition in \cite{CKW, CKW2} rules out the case that the metric measure space
  $(\sX, d, m)$ has atoms. In a recent paper \cite{Ma}, the above stable characterization of heat kernels has been 
  extended to metric measure spaces allowing atoms  by relaxing (RVD) condition   to a quasi-reverse volume doubling condition (QRVD), 
  through introducing an auxiliary space that reduces (QRVD)  to the (RVD) case. 
    
  \begin{definition}\label{Ahlfors}\rm 
  For a metric measure space $(\sX, d, m)$, 
  we say  a subset $D\subset \sX$ is {\it Ahlfors regular}  if there is a constant $c_D\in (0,1]$ such that 
\begin{equation}\label{e:ahlfors}
m(B (x,r) \cap D) \geq c_D m(B(x,r))  \quad  \hbox{ for every }x\in D \hbox{ and } 0<r<{\rm diam}(D)/2. 
\end{equation} 
\end{definition}

Here ${\rm diam}(D):=\sup_{x, y\in \sX} d(x, y)$ stands for the diameter of $D$. 
Clearly, $D\subset \sX$ is Ahlfors regular if and only if $m(B (x,r) \cap D)\asymp m(B(x, r)$ for any $x\in \sX$
and $0<r<{\rm diam}(D)/2$. As far as we know, this definition of Ahlfors regular sets on metric measure spaces is introduced 
for the first time. It is consistent with the definition of Ahlfors $d$-regular set in $\R^d$
where the Lebesgue volume of the ball $B(x, r)$ is $c_d r^d$. 
Ahlfors regular sets on a metric measure space  $(\sX, d, m)$  have some desired properties. 
We show in Lemma \ref{L:2.5} below that if $(\sX, d, m)$ is (VD)
 and $D \subset \sX$ is an Ahlfors regular open subset,
then $m(\partial D)=0$
and $(\overline D, d, m_0)$ is (VD), where    $m_0:=m|_D$,
 that is,  $m_0(A):=m(A\cap D)$ for any Borel subset $A\subset \overline D$;
 if $(\sX, d, m)$ is (QRVD), then so is $(\overline D,d,m_0)$. 
 Here $\bD$ is the closure of $D$ in $\sX$ and $\partial D:=\bD \setminus D$ is its boundary.
 
Suppose now that the metric measure space $(\sX,d,m)$ satisfies {\rm(VD)} and {\rm(QRVD)}, 
and $(\sE, \sF)$ is a purely non-local regular symmetric
Dirichlet form on $L^2(\sX; m)$ that satisfies the two-sided  heat kernel estimates ${\bf HK}(\phi)$
(see Definition \ref{D:2.6}(b) for its definition).
For an open subset $D\subset \sX$, denote by $(\bar \sE, \bar \sF)$ the active reflected Dirichlet form
on  $L^2(\overline D; m_0)$ of $(\sE, \sF)$ on $\overline D$ in the sense of \cite{C, CF}.
The main result of this paper is Theorem \ref{mainthm}. 
It asserts that for any Ahlfors regular open subset $D$ of $\sX$, 
the   metric measure Dirichlet space $(\bD,d,m_0,\bar \sE, \bar \sF)$ is regular and has two-sided heat kernel estimates ${\bf HK}(\phi)$. 
We know from \cite{CKW, Ma} that under (VD) and (QRVD), ${\bf HK}(\phi) \Longleftrightarrow {\bf J}_\phi+{\rm CSJ}(\phi)$.
Since ${\bf J}_\phi$ condition for the active reflected Dirichlet form $(\bar \sE, \bar \sF)$ on  $L^2(\overline D; m_0)$  is inherited from that 
of $(\sE, \sF)$, in view of the result of Lemma \ref{L:2.5} mentioned above, it suffices to show that 
${\rm CSJ}(\phi)$ for $(\bar \sE, \bar \sF)$ is inherited from that of $(\sE, \sF)$ as well. 
Since the cut-off Sobolev inequality ${\rm CSJ}(\phi)$ is formulated in terms of energies and $L^2$-norms (see Definition \ref{D:2.6}(e) below), 
it is natural to try to show that there is an extension operator from $D$ to $\sX$ that preserves the energy and $L^2$ bounds.
This approach indeed works and is carried out in this paper.

When $(\sE,\sF)$ is a strongly local regular symmetric Dirichlet form on $L^2(\sX; m)$ that satisfies a two-sided sub-Gaussian heat kernel bounds
and $D\subset \sX$ is a uniform domain, 
such an extension operator has been constructed in Murugan \cite{Mathav} which is used to establish the cutoff Sobolev inequality for the reflected diffusion on $\overline D$. It follows from the analogous analytic characterization of sub-Gaussian heat kernel bounds for diffusions that the reflected diffusion on $\overline D$ has a two-sided sub-Gaussian estimates. This result has recently been extended to inner uniform domains
in \cite{Ant}. 

There are major differences between the approach of constructing extension operators in this paper for reflected jump processes and 
that of \cite{Mathav} for reflected diffusions. 

    Ahlfors regular open sets are much more general than uniform and inner uniform domains. 
They do not satisfy a corkscrew condition that uniform and inner uniform domains enjoy, 
and hence a `reflection' in the sense of \cite[Proposition 3.12]{Mathav} 
from the Whitney cover of exterior domain to the Whitney cover of interior domain does not work. 
Instead, we introduce in Proposition \ref{P:3.3} a class of functions that plays the role of the Whitney cover of the domain.

Another difference is we need to estimate the long distance energy  as our Dirichlet form is non-local. We need to split the energy into three parts, near diagonal in the exterior domain, off diagonal in the exterior domain, and the cross energy, and carefully estimate each part in Section 3. 

The rest of paper is organized as follows. In Section \ref{S:2}, we state the main theorem (Theorem \ref{mainthm}) and related conditions. 
In Section \ref{S:3}, we construct an extension operator and show the local scale-invariant estimates on $L^2$ norm. 
In Section \ref{S:4}, we prove the local scale-invariant estimates of the extension operator on energy, and in particular, this implies the  active reflected Dirichlet form is regular. Finally, in Section \ref{S:5}, we establish the cut-off Sobolev inequality for the active
reflected Dirichlet form, and derive the two-sided heat kernel estimates for the reflected jump process on the closure of Ahlfors regular open subsets.

\section{Conditions and main theorem}\label{S:2} 
Throughout the paper, $(\sX,d)$ is a locally compact metric measure space, and $m$ is a Radon measure on $(\sX,d)$ with full support and may possibly have atoms. We let $B(x,r):=\{y\in\sX:\,d(x,y)<r\}$ denote the open ball of radius $r$ centered at $x$ in $\sX$. For short, we write 
\[
V(x,r)=m(B(x,r))\ \hbox{ for }x\in \sX\hbox{ and }r>0
\]
from  time to time. We use $:=$ as a way of definition. 
We use  $d(A,B):=\inf\{d(x,y):\,x\in A,\,y\in B\}$ to denote the distance between two sest $A,B\subset \sX$, and 
define  $d(x,A):=d(\{x\},A)$ for $x\in \sX$ and $A\subset \sX$. 
We denote by $\sB (\sX)$ the collection of all Borel measurable subsets of $\sX$, while 
for $F\in \sB(\sX)$, $\sB (F)$ stands for the collection of all Borel measurable subsets of $ F$.

For a set $E\subset\sX$,  denote by $E^c:=\sX\setminus E$ its  complement  in $\sX$,  $\overline{E}$  its closure,  and
  $\partial E=\overline{E}\cap\overline{E^c}$  the boundary of $E$.   Denote by $\1_E$  the indicator function of $E\subset\sX$, that is, 
\[
\1_E(x)=\begin{cases}1\quad\hbox{ for }x\in E,\\0\quad\hbox{ for }x\in\sX\setminus E.\end{cases}
\]

\begin{definition}
We say a metric measure space $(\sX,d,m)$ is volume doubling {\rm(VD)} if there is $C_1\in (1,\infty)$ such that
$$
V(x,2r) \leq C_1 \, V(x,r) \qquad \hbox{ for every }  x\in \sX   \hbox{ and }   r\in (0,\infty),
$$
where $V(x, r):= m(B(x, r))$. 
This is equivalent to the existence of positive constants $c_1$ and $d_1$ so that 
\begin{equation}\label{e:vd}
	\frac{ V(x, R)}{V(x, r)} \leq c_1  \Big(\frac{R}{r} \Big)^{d_1}  \quad 
	\hbox{ for every }x\in \sX   \hbox{ and }   0<r\leq R<\infty.
\end{equation}
We say that reverse volume doubling property {\rm(RVD)} holds if there are $\lambda_0\geq 2 $ and $C_2 >1$ so that 
\begin{equation}\label{e:rvd}
	V(x, \lambda_0 r)   \geq C_2  V(x, r) \quad \hbox{ for every }x\in \sX   \hbox{ and }   0<r  < {\rm diam}(\sX)/\lambda_0.
\end{equation}
This is equivalent to the existence of positive constants $c_2$ and $d_2$ so that 
$$
\frac{V(x,R)} {V(x,r)} \geq  c_2 \Big(\frac{R}{r} \Big)^{d_2}  \quad 
\hbox{ for every }x\in \sX   \hbox{ and }   0<r \leq R < {\rm diam}(\sX).
$$
\end{definition}\smallskip 

(RVD) condition implies that $m(\{x\})=0$ for every $x\in \sX$.
In a recent paper \cite{Ma}, the following quasi-reverse volume doubling (QRVD) condition  
is introduced in place of (RVD)
to extend  the stability results in \cite{CKW} on heat kernel estimates for pure jump symmetric Markov processes 
to some metric measures spaces having  atoms including graphs. 

\begin{definition}
We say a metric measure space $(\sX,d,m)$ is quasi-reverse volume doubling {\rm(QRVD)} if there are constants $\lambda_0\geq 2 $ and $C>1$ so that 
\[
V(x,\lambda_0r)\geq CV(x,r)\ \hbox{ for }x\in\sX\hbox{ and }d(x,\sX\setminus\{x\})\leq r<\diam(\sX)/\lambda_0. 
\] 
\end{definition}

\medskip

It is clear that (RVD) is equivalent to (QRVD) when $(\sX,d)$ does not have an isolated point. 

\smallskip  

\subsection{Consequences of (VD) and (QRVD)}

\begin{lemma}\label{lemmavd1}
Suppose that $(\sX, d, m)$ is {\rm(VD)}, and  $c_1$ and $d_1$ are the constants in \eqref{e:vd}. Then 
\begin{equation}\label{e:vd1}
\frac{V(x,r_1)}{V(y,r_2)}\leq c_1 \Big(\frac{d(x,y)+r_1}{r_2} \Big)^{d_1}\quad\hbox{ for }x,y\in\sX\hbox{ and }r_1,r_2>0. 
\end{equation} 
\end{lemma}

\begin{lemma}\label{lemmavd2}
Suppose that $(\sX, d, m)$  is {\rm(VD)}, and  $c_1$ and $d_1$ are the constants in \eqref{e:vd}. Let $E$ be a bounded subset of $\sX$ and $r>0$. 
There is a finite subset $\{x_i\}_{i=1}^N\subset E$ with $ N\leq c_1(1+4\,\diam(E)/r)^{d_1}$ such that $d(x_i,x_j)>r$ for $1\leq i<j\leq N$ and $E\subset\bigcup_{i=1}^N B(x_i,r)$.
 Moreover, the following properties hold. 
\begin{enumerate}[\rm (a)]
\item $\sum_{i=1}^N\1_{B(x_i,\eta r)}\leq c_1(2\eta+1)^{d_1}$ for every $\eta\geq1$.

\item $\bigcup_{x\in E}B(x,(\eta-1)r)\subset\bigcup_{i=1}^NB(x_i,\eta r)$ for every $\eta>1$. 
\end{enumerate}
\end{lemma}
\begin{proof}
 The first statement is well-known for a doubling metric measure space $(\sX,d,m)$ (see \cite[Section 10.13]{He}). 
 For the convenience of the reader, we provide a detailed proof here. 	
We find the finite set of points $\{x_i\}_{i=1}^{N}$ by the following procedure. First, we pick $x_1\in E$. Next, if $E\subset B(x_1,r)$, 
we are done by ending the process with $\{x_1\}$; otherwise, we pick $x_2\in E\setminus B(x_1,r)$ to form a larger set $\{x_i\}_{i=1}^2$. Next, we repeat the procedure for the set $\{x_i\}_{i=1}^2$. If $E\subset \bigcup_{i=1}^2B(x_i,r)$, we end the process; otherwise, we pick $x_3\in E\setminus\bigcup_{i=1}^2B(x_i,r)$ to form  $\{x_i\}_{i=1}^3$. We keep doing this until $E\subset \bigcup_{i=1}^{N}B(x_i,r)$. This process has to stop after finitely many steps and $N\leq  c_1(1+4\diam(E)/r)^{d_1}$. This is because  $\{B(x_i,r/2);i\geq 1\}$ are pairwise disjoint, and for $i\geq 1$,  
\[
B(x_i, r/2)\subset B(x_1,\diam(E)+r/2)\subset B(x_i,2\diam(E)+r/2),
\] 
so by \eqref{e:vd}, 
\begin{align*}
 N\cdot V(x_1,\diam(E)+r/2)
&\leq \sum_{i=1}^{N} V(x_i,2\diam(E)+r/2)  \\
&\leq c_1(1+4\,\diam(E)/r)^{d_1}\sum_{i=1}^{N}V(x_i,r/2)\\
&\leq c_1(1+4\,\diam(E)/r)^{d_1}V(x_1,\diam(E)+r/2). 
\end{align*} 
This proves the existence of finite many points  $\{x_i\}_{i=1}^N$ in $E$ with the desired property with $N\leq c_1(1+4\,\diam(E)/r)^{d_1}$.

(a). This part can be proved by  a similar idea.
Let $x\in\sX$. For each $1\leq i\leq N$ such that $x\in B(x_i,\eta r)$, 
\[
B(x_i,r)\subset B(x,(\eta+1)r)\subset B(x_i,(2\eta+1)r),
\]
so by \eqref{e:vd},
\begin{align*}
\#\{1\leq i\leq N:x\in B(x_i,\eta r)\}\cdot V(x,(\eta+1)r)&
\leq \sum_{1\leq i\leq N:x\in B(x_i,\eta r)}V(x_i,(2\eta+1)r)\\
\\&\leq c_1(2\eta+1)^{d_1}\sum_{1\leq i\leq N:x\in B(x_i,\eta r)}V(x_i,r)
\\&\leq c_1(2\eta+1)^{d_1}V(x,(\eta+1)r).
\end{align*}
It follows that $\sum_{i=1}^N\1_{B(x_i,\eta r)}(x)=\#\{1\leq i\leq N:x\in B(x_i,\eta r)\}\leq c_1(2\eta+1)^{d_1}$.

(b). For each $x\in E$, we can find $1\leq i\leq N$ so that $x\in B(x_i,r)$, and hence $B(x,(\eta-1)r)\subset B(x_i,\eta r)$.
\end{proof}

\smallskip

For notational convenience, for any open set $D\subset \sX$,  we write $B_D(x,r) :=D\cap B(x,r)$.

 \begin{lemma}\label{L:2.5}
Suppose that  $D\subset \sX$ is Ahlfors regular. 
\begin{enumerate}[(a)]
\item If $(\sX,d,m)$ is {\rm(VD)}, then $m(\partial D)=0$ and thus  $m_0= m|_{\overline D} $.
\item If $(\sX,d,m)$ is {\rm(VD)}, then so is $(\bD,d,m_0)$.
\item If $(\sX,d,m)$ is {\rm(QRVD)}, then so is $(\bD,d,m_0)$. 
\end{enumerate}
\end{lemma}

\begin{proof}
(a).  Our proof uses the same idea as that for \cite[Lemma 3.5]{Mathav}.  For each $x\in \partial D$ and $0<r<\diam(D)$, take $y\in B_D(x,r/2)$. Clearly, 
\[
B(y,r/2)\subset B(x,r)\subset B(y,2r).
\]
Hence
\begin{align*}
	\frac{m(B(x,r)\cap\partial D)}{V(x,r) }&\leq1-\frac{m(B(x,r) \cap D)}{V(x,r)}\leq 1- \frac{m(B(y,r/2) \cap D)}{V(y, 2r)}\\
	&\leq 1-c_D {\frac{V (y,r/2) }{V(y, 2r)}} \leq 1-C_1
\end{align*}
for some $C_1$ depending only the parameters of {\rm (VD)} and the constant $c_D$ in \eqref{e:ahlfors}.
  Hence, by the Lebesgue differentiation theorem \cite[Theorem 1.8]{He}, 
\[
\1_{\partial U}(x)=\lim\limits_{r\to 0}\frac{m(B(x,r)\cap\partial D)}{m(B(x,r))}\leq 1-C_1<1
\quad \hbox{ for }m\hbox{-a.e.}\  x\in \partial U.
\]
This implies $m(\partial U)=0$. 

\smallskip

 (b). For $x\in \overline D$ and $r<\diam(D)$
 , let $y\in D$ be such that $d(x,y)<r/2$.
Then  by the (VD) of $(\sX,d,m)$   and the Ahlfors regularity of $D\subset \sX$,
\begin{align*}
m_0(B(x,2r))\leq V (x,2r) \leq  V(y, 3r)\leq C_1 V(y, r/2)\leq C_2 m_0 (B(y, r/2) ) \leq C_2 m_0(B(x, r)) ,
\end{align*}
where $C_2=C_1/c_D$.

\smallskip

(c). Let $\lambda_0 \geq 2$ and $C_3>1$ be the constants of (QRVD) for $(\sX,d,m)$, and $c_D\in (0, 1)$ the constant in
\eqref{e:ahlfors}.  
Without loss of generality, we assume $1/c_D > C_3$. 
 Observe that $d(x, \overline D\setminus \{x\})\geq  d(x, \sX\setminus \{x\})$. 

For each $x\in D$ and $d(x,\overline{D}\setminus \{x\})\leq r<\diam(D)/\lambda_0$, by  the Ahlfors regularity of $D\subset \sX$ and  (QRVD) of $(\sX,d,m)$,
\[
m_0(B(x,\lambda_0r))\geq c_D^{-1} \,V(x,\lambda_0 r)\geq C_4 V(x,r) \geq C_4m_0(B(x,r)),
\]
where $C_4=(c_D\cdot C_3)^{-1}>1$. For $x\in \partial D$ and $0<r<\diam(D)/\lambda_0$, we choose $y\in D$ such that $d(x,y)<r/4$. Note that $d(y,\overline{D}\setminus \{y\}) \leq d(x, y) <r/4$. Thus we have from the above that 
\[
m_0(B(x,\lambda_0r))\geq m_0(B(y,\lambda_0r/2))\geq C_4 m_0(B(y,r/2))\geq C_4m_0(B(x,r/4)). 
\] 
This established the (QRVD) property of  $(\bD,d,m_0)$.
\end{proof}

\subsection{Main result of this paper}

In the rest of this paper, $(\sE,\sF)$ is a symmetric regular Dirichlet form on $L^2(\sX; m)$
of pure jump type with jump measure $J(dx,dy)=J(x,y)m(dx)m(dy)$, where 
\begin{equation}\label{jumpDirichlet}
\begin{split}
\sF& :=  \left\{ u\in L^2 (\sX; m): \int_{\sX\times \sX} (u(x)-u(y))^2 J(dx, dy)  <\infty \right\}, \\
\sE(u, v)& := \frac12  \int_{\sX\times \sX} (u(x)-u(y)) (v(x)-v(y))  J(dx, dy)\quad \hbox{for } u, v \in \sF.		
\end{split}
\end{equation}
In  this paper, we always represent a function $f\in \sF$ by its $\sE$-quasi-continuous version, which is unique up to an $\sE$-polar set. 
It is well known that there exists an $m$-symmetric Hunt process $X$ on $\sX\setminus \sN$ associated with 
$(\sE, \sF)$, where $\sN\subset \sX$ is a Borel measurable properly exceptional set of $X$ which in particular is $\sE$-polar.
 We call $(\sX, d, m, \sE, \sF)$ a metric measure Dirichlet space (MMD).
We refer the reader to  \cite{CF, FOT} for the above facts as well as other basic properties and terminologies of Dirichlet forms.

Let $\phi:[0,\infty)\to [0,\infty)$ be  a strictly increasing function such that $\phi(0)=0$, $\phi(1)=1$ and there exist constants $c_1,c_2>0$ and $\beta_2\geq \beta_1>0$ such that 
\begin{equation}\label{eqnphi}
c_1(\frac{R}{r})^{\beta_1}\leq \frac{\phi(R)}{\phi(r)}\leq c_2(\frac{R}{r})^{\beta_2}\ \hbox{ for all }0<r\leq R<\infty.
\end{equation}

\begin{definition}\label{D:2.6}  \rm 
\begin{enumerate}[\rm (a)]
\item We say ${\bf J}_{\phi}$ holds if 
\[
\frac{C_1}{V(x,d(x,y))\phi(d(x,y))}\leq J(x,y)\leq \frac{C_2}{V(x,d(x,y))\phi(d(x,y))}\ \hbox{ for every }x\not= y\in \sX.
\]
We say that ${\bf J}_{\phi, \leq}$ (resp. ${\bf J}_{\phi, \geq}$) holds if the upper bound (resp. lower bound) in the above formula holds.
		
\item Denote by $X$ the $m$-symmetric Hunt process on $\sX$ associated with $(\sE,\sF)$ on $L^2(\sX;m)$. We say the process $X$ satisfies two-sided heat kernel estimates of mixed type ${\bf HK}(\phi)$ if 
 it has a transition density function $p(t,x,y)$ on $(0, \infty) \times (\sX \setminus \sN) \times (\sX \setminus \sN)$
 with respect to the measure $m$ so that for every $(t, x, y) \in (0, \infty) \times (\sX \setminus \sN) \times (\sX \setminus \sN)$,
\begin{align}
	c_1 \Big( \frac{1}{V(x,\phi^{-1}(t))}\wedge  &\frac{t}{V(x,d(x,y))\phi(d(x,y))}  \Big) \nonumber\\
	& \, \leq p(t,x,y)  \, \leq \, c_2 \Big(  \frac{1}{V(x,\phi^{-1}(t))}\wedge \frac{t}{V(x,d(x,y))\phi(d(x,y))}  \Big). \label{e:HK}
\end{align}
 We say ${\bf UHK}(\phi)$  holds if the upper bound in \eqref{e:HK}  holds.

\item We define the relative capacity for $A\subsetneq B\subsetneq\sX$ by
\[
{\rm Cap}(A,B):=\inf\{\mcE(f,f):\,f\in\sF,\,f|_A=1,\,f|_{\sX\setminus B}=0\}. 
\]
We say that condition {\rm Cap}$(\phi)_{\leq}$ holds for $(\mcE,\mcF)$ if 
there is $c>0$ such that for every $x\in \sX$ and $r<\diam(\sX)/3$,
\begin{equation*} 
\operatorname{Cap}\big(B(x,r),B(x,2r)\big)\leq c\frac{m(B(x,r))}{\phi(r)}.  
\end{equation*}

\item Let $U\subset V$ be open sets of $\sX$ with $U\subset\overline{U}\subset V$. We say a non-negative bounded measurable function $\varphi$ is a cut-off function for $U\subset V$, if $\varphi=1$ on $U$, $\varphi=0$ on $\sX\setminus V$ and $0\leq \varphi\leq 1$ on $\sX$. 

\item We say condition ${\rm CSJ}(\phi)$ holds for $(\mcE,\mcF)$ if there exist constants $c,C_0\in (0,1]$ and $C_1,C_2>0$ such that for every $0<r\leq R<c\,\diam(\sX)$, almost all $x_0\in \sX$ and any $f \in \sF$, there exists a cut-off function $\varphi \in \mcF$ for $B(x_0,R)\subset  B(x_0,R+r)$ so that the following holds:
\begin{align*}
	&\quad\  \int_{B (x_0,R+(1+C_0)r)}f^2d\Gamma(\varphi,\varphi) \\
	&\leq  C_1\int_{U\times U^*}\big(f(x)-f(y)\big)^2J(dx,dy)+\frac{C_2}{\phi(r)}\int_{B (x_0,R+(1+C_0)r)}f(x)^2 m(dx),
\end{align*}
where  $U=B(x_0,R+r)\setminus B(x_0,R)$, $U^*=B(x_0,R+(1+C_0)r)\setminus B(x_0,R-C_0r)$ and
\[
\Gamma(u,v)(dx)=\int_\sX \big(u(x)-v(y)\big)\big(u(x)-v(y)\big)J(dx,dy)\quad\hbox{ for }u,v\in\sF.
\]
\end{enumerate}
\end{definition}

\begin{remark} \label{R:2.7} \rm 
\begin{enumerate}[\rm (i)]
\item  By \cite[Proposition 3.1(2)]{CKW}, the lower bound in ${\bf HK}(\phi)$ implies that 
 $(\sE, \sF)$ is conservative, that is,
the process $X$ has infinite lifetime $\bP_x$-a.s. for every 
$x\in \sX\setminus\mathcal{N}$. Moreover, by \cite[Theorem 1.13 and Lemma 5.6]{CKW}, (VD) and ${\bf HK}(\phi)$ 
imply that the transition density function $p(t, x, y)$ is jointly H\"older continuous in $ (x, y) $
for every $t > 0$ and so the two-sided estimates \eqref{e:HK} holds for all $x, y \in \sX$. 
This in particular implies that, under (VD) and ${\bf HK}(\phi)$,
 the Hunt process $X$ can be refined to start from every point in $\sX$, and, by 
 \eqref{e:HK},  $X$ is a Feller process having strong Feller
 property on $\sX$. Furthermore, by \cite[Theorem 1.18 and Corollary 1.3]{CKW2} and \cite[Theorems 1.20 and 1.21]{Ma}, under 
  (VD), (QRVD) and ${\bf HK}(\phi)$, the transition density function $p(t, x, y)$ of $X$ is jointly H\"older continuous
  in $(t, x, y)$ on $ (0, \infty)\times \sX \times \sX$. 
 
\item
 By \cite[Proposition 3.3]{CKW}, under {\rm (VD)}, 
\begin{equation}\label{e:UHK}
 {\bf UHK}(\phi) + \hbox{$(\sE, \sF)$ is conservative} \Longrightarrow  {\bf J}_{\phi,\leq}
\quad \hbox{and} \quad   {\bf HK}(\phi)   \Longrightarrow  {\bf J}_{\phi}.
\end{equation}
We know 
from   Propositions  3.6 and Definition 1.5(ii) of \cite{CKW} that 
\begin{equation}\label{e:CSJ}
{\rm (VD)} +  {\bf UHK}(\phi) + \hbox{$(\sE, \sF)$ is conservative}  \Longrightarrow {\rm CSJ}(\phi) , 
\end{equation}
and 
from \cite[Proposition 2.3]{CKW} that 
\begin{equation} \label{e:Cap}
{\rm (VD)}+ {\bf J}_{\phi,\leq}  + {\rm CSJ}(\phi) \Longrightarrow {\rm Cap}(\phi)_{\leq} .
\end{equation}
It is shown  in \cite[Theorem 1.20]{Ma} that \cite[Theorem 1.13]{CKW} holds with condition (RVD)
being relaxed to (QRVD) by constructing  an  auxiliary space  to  smooth out the atoms.
 That is, the following equivalence holds when $(\sX, d)$ is unbounded: 
\begin{equation}\label{e:2.6a}
	\rm (VD)+(QRVD)+{\bf HK}(\phi)\Longleftrightarrow  (VD)+(QRVD)+{\bf J}_\phi+{\rm CSJ}(\phi).
\end{equation}
 
 \item Although it is assumed in \cite{CKW} that the state space  is unbounded,  the results there hold for bounded state spaces as well with some minor modifications  and  also some simplifications; see \cite[Remark 8.3]{CC} for details. 
 Using the construction of an auxiliary space  that smooths out the atoms as in \cite{Ma}, 
 the equivalence \eqref{e:2.6a} holds for bounded $(\sX, d)$ as well. 
 
 \item Building on recent progress in the local setting due to S.~Eriksson-Bique \cite{E}, 
 in a very recent preprint,  Murugan showed  \cite[Theorem 1.6]{Mathav2} that 
 \begin{equation} \label{e:Cap2}
  {\rm (VD)}+ {\bf J}_{\phi} + {\rm Cap}(\phi)_{\leq}  \Longrightarrow  {\rm CSJ}(\phi) 
 \end{equation}
 This together with \eqref{e:Cap} and \eqref{e:2.6a} yields 
 \begin{eqnarray}
	\rm (VD)+(QRVD)+{\bf HK}(\phi)
	&\Longleftrightarrow & {\rm (VD)+(QRVD)} +{\bf J}_\phi+{\rm CSJ}(\phi)  \nonumber \\
	&\Longleftrightarrow & {\rm (VD)+(QRVD) } +{\bf J}_\phi+{\rm Cap}(\phi)_{\leq}.  \label{e:Cap3}
\end{eqnarray}
In this paper, we do not use the implication \eqref{e:Cap2} nor the second equivalence in \eqref{e:Cap3}.
\end{enumerate}
\end{remark}

\medskip

Let $D\subset\sX$ be an open subset. 
  Let $m_0$ be the Radon measure on $\bD$ defined by
$$
m_0(A)=m(A\cap D) \quad \hbox{for any  }  A\in \sB ( \bD). 
$$
Let $\sF^D :=\{u\in \sF: u=0 \hbox{ $\sE$-q.e. on } \sX \setminus D\}$.
It is well known \cite{CF, FOT}  that $(\sE, \sF^D)$ is the Dirichlet form on $L^2(D; m_0)$ for the part process $X^D$ of $X$ killed upon leaving $D$.
 Denote by $({\bar \sE}, {\bar \sF})$ the  active reflected Dirichlet space of $(\sE, \sF^D)$, that is, 
 \begin{align}
{\bar \sF}&:= \left\{ u\in L^2(\bD;m_0): \int_{D\times D} (u(x)-u(y))^2 J(dx, dy) <\infty \right\}, \label{e:2.8} \\
{\bar \sE}(u, v)&:=\frac12  \int_{D\times D} (u(x)-u(y)) (v(x)-v(y))  J(dx, dy)\quad \hbox{for } u, v \in {\bar \sF} ; \label{e:2.9}
\end{align}
see \cite{C,CF}. It is shown in \cite{C} that $(\bar\sE, \bar\sF)$ is always a Dirichlet form on $L^2(\overline D; m_0)$. We call $(\overline D, d, m_0, \bar \sE, \bar\sF)$ a reflected metric measure Dirichlet space.
Note that $(\bar \sE, \bar \sF)$ depends on $D$ but for notational simplicity, we do not indicate  $D$ in the notation.
When $D=\sX$,  $(\bar \sE, \bar \sF)$ is just $(\sE, \sF)$. 
  The goal of the paper is to find a sufficient condition for a proper open subset $D$ of $\sX$  so that   $({\bar \sE}, {\bar \sF})$  
  is a regular Dirichlet form on $L^2(\bar D; m_0)$ and it has two-sided heat kernel bounds ${\bf HK}(\phi)$.

\medskip

The following is  the main theorem of this paper. 

\begin{theorem}\label{mainthm}
Suppose  that $(\sX,d,m,\mcE,\mcF)$ satisfies {\rm(VD)}, {\rm(QRVD)} and ${\bf HK}(\phi)$, and that $D$ is an Ahlfors regular open subset of $\sX$. Then the reflected MMD $(\bD,d,m_0,\bar \sE, \bar \sF)$ is regular on $L^2(\overline D; m_0)$
and satisfies {\rm(VD)}, {\rm(QRVD)} and ${\bf HK}(\phi)$. 
\end{theorem}

\section{An extension operator}\label{S:3}
In this section, we construct an extension operator $\Ex:L^2(\bD;m_0)\to L^2(\sX;m)$ that has good properties on energy and $L^2$-norm.  This is a key step in establishing the ${\rm CSJ}(\phi)$ property 
for  active
 reflected Dirichlet space $(\bar \sE, \bar \sF)$ of \eqref{e:2.8}-\eqref{e:2.9} on $L^2(\bar D; m_0)$ when $D\subset \sX$ is Ahlfors regular.
Our method uses a Whitney cover of $\sX\setminus \overline{D}$. By \cite[Proposition 3.2]{Mathav}, for any non-empty open subset 
$U\subsetneq\sX$ and $\varepsilon\in (0,1/2)$, there is an $\varepsilon$-Whitney cover of $U$ in the following sense. 

\begin{definition}\label{D:3.1} \rm
 Let $\varepsilon\in (0,1/2)$ and $U$ be a non-empty open proper subset of $\sX$. We say   a collection $\{B(x_i,r_i):x_i\in U,r_i>0,i\in I\}$ of open balls in $U$ is an $\varepsilon$-Whitney cover of $U$ if it satisfies the following properties.
\begin{enumerate}[\rm (i)]
\item The collection of open balls $\{B(x_i,r_i):\,i\in I\}$ are pairwise disjoint.
 		
\item $r_i=\frac{\varepsilon}{1+\varepsilon}d(x_i,\sX\setminus U)$ for all $i\in I$. 
 		
\item $\bigcup_{i\in I}B\big(x_i,2(1+\varepsilon)r_i\big)=U$.  
\end{enumerate}
\end{definition}

Let $D\subset \sX$ be an open subset so that $\overline D \subsetneq\sX$.
In the rest of this paper, we fix a $\frac{1}{4}$-Whitney cover  $\{B(x_i,r_i):x_i\in \sX\setminus \overline{D},r_i>0,i\in I\}$ of $\sX\setminus \overline{D}$. 
For notational simplicity,  we write for each $i\in I$ and $\lambda>0$, 
\begin{equation}\label{e:3.1a}
B_i :=B(x_i,r_i) \quad \hbox{ and } \quad \lambda B_i :=B(x_i,\lambda r_i) .
\end{equation} 
By Definition \ref{D:3.1}, we have the following properties. 
\begin{enumerate}[(D.i)]
\item The collection of sets $\{B_i:\,i\in I\}$ are pairwise disjoint.

\item $d(x_i, \overline D )=5r_i$ for each $i\in I$.

\item $\sX\setminus\overline{D}=\bigcup_{i\in I}\frac52B_i$.
\end{enumerate}

Let
\[\Lambda := \{ i\in I:0<r_i<\diam(D)/2\}.\medskip \]

 The following lemma gives some basic properties of this Whitney cover. 

\begin{lemma}\label{L:3.2}
Suppose that $(\sX, d, m)$ is {\rm(VD)}.
\begin{enumerate}[\rm (a)]
\item If $1<\lambda<5$ and $\lambda B_i\cap \lambda B_j\neq \emptyset$ for some $i,j\in I$, then 
\[
\frac{5-\lambda}{5+\lambda}r_i\leq r_j\leq \frac{5+\lambda}{5-\lambda}r_i.
\] 
		
\item For $1<\lambda<5$, there is a positive constant $C$ depending on $\lambda$ and the parameters of \eqref{e:vd} so that $\#\{j\in I:\,\lambda B_j\cap \lambda B_i\neq \emptyset\}\leq C$ for every $i\in I$.  
		
\item If $x\in \lambda B_i$ for some $i\in I$ and $0<\lambda<5$, then 
\[
\frac1{5+\lambda}d(x,\bD)\leq r_i\leq \frac1{5-\lambda}d(x,\bD).
\] 
		
\item For $0<\lambda<5$, there is a positive constant $C$ depending on $\lambda$ and the parameters of \eqref{e:vd} so that $
\#\{i\in I:\,x\in\lambda B_i\}\leq C$ for every $x\in \sX\setminus\bD$. 

\item Suppose that $D$ is bounded. Then, for every $r,s>0$, 
\[
\#\{i\in I:\,r_i<r,\,\lambda B_i\setminus (\cup_{x\in \overline{D}}B(x,s))\neq\emptyset\}<\infty. 
\]
\end{enumerate}
\end{lemma}
\begin{proof}
	(a). Since $\lambda B_i\cap \lambda B_j\neq \emptyset$, we have $d(x_i,x_j)\leq \lambda(r_i+r_j)$. Hence, by (D.ii),
	\[
	5r_j=d(x_j,\bD)\leq d(x_i,\bD)+d(x_i,x_j)\leq 5r_i+(\lambda r_i+\lambda r_j).
	\]
	It follows immediately that $r_j\leq\frac{5+\lambda}{5-\lambda}r_i$. Symmetrically, we also have $r_i\leq \frac{5+\lambda}{5-\lambda}r_j$.
	
	(b). By (a), if $\lambda B_i\cap\lambda B_j\neq\emptyset$, then 
	\[
	B_j\subset B(x_i,\lambda r_i+\lambda r_j+r_j)\subset B(x_i,\lambda r_i+\lambda\frac{5+\lambda}{5-\lambda}r_i+\frac{5+\lambda}{5-\lambda}r_i)=\frac{5+11\lambda}{5-\lambda}B_i
	\] 
	and symmetrically $B_i\subset\frac{5+11\lambda}{5-\lambda}B_j$. 
	Hence by (D.i) and \eqref{e:vd}, 
	\begin{align*}
		&\quad\ C_1m(B_i)\geq m(\frac{5+11\lambda}{5-\lambda}B_i)\geq \sum_{j\in I:\,\lambda B_i\cap\lambda B_j\neq\emptyset}m(B_j)\\
		&\geq \sum_{j\in I:\,\lambda B_i\cap\lambda B_j\neq\emptyset}C_1^{-1}m(\frac{5+11\lambda}{5-\lambda}B_j)\geq C_1^{-1}\#\{j\in I:\,\lambda B_j\cap \lambda B_i\neq \emptyset\}m(B_i),
	\end{align*} 
	where $C_1=c_1(\frac{5+11\lambda}{5-\lambda})^{d_1}$.
	Hence, $\#\{j\in I:\,\lambda B_j\cap \lambda B_i\neq \emptyset\}\leq C_1^2$. 
	
	(c). This follows immediately from the following observation by (D.ii): 
	\[
	(5-\lambda)r_i\leq -d(x,x_i)+d(x_i,\bD)\leq d(x,\bD)\leq d(x,x_i)+d(x_i,\bD)\leq (5+\lambda)r_i.
	\]
	
	(d). By (c), if $x\in \lambda B_i$, 
	\begin{align*}
		B_i&\subset B(x,\lambda r_i+r_i)\subset B(x,\frac{1+\lambda}{5-\lambda}d(x,\bD))\\
		&\subset B(x_i,\lambda r_i+\frac{1+\lambda}{5-\lambda}d(x,\bD))\subset B(x_i,\lambda r_i+\frac{(1+\lambda)(5+\lambda)}{5-\lambda}r_i)=\frac{5+11\lambda}{5-\lambda}B_i.
	\end{align*}
	Hence, by \eqref{e:vd}, 
	\[
	C_2\,m(B_i)\geq m(\frac{5+11\lambda}{5-\lambda}B_i)\geq m(B(x,\frac{1+\lambda}{5-\lambda}d(x,\bD))),
	\]
	where $C_2=c_1(\frac{5+11\lambda}{5-\lambda})^{d_1}$. Then, by (D.i),
	\begin{align*}
		m(B(x,\frac{1+\lambda}{5-\lambda}d(x,\bD)))
		&\geq\sum_{i\in I:\,x\in \lambda B_i}m(B_i)\\
		&\geq \#\{i\in I:\,x\in\lambda B_i\}C_2^{-1} m(B(x,\frac{1+\lambda}{5-\lambda}d(x,\bD))).
	\end{align*}
	Hence, $\#\{i\in I:\,x\in\lambda B_i\}\leq C_2$. 
	
	(e). We fix $z\in D$, and we write $J=\{i\in I:\,r_i<r,\,\lambda B_i\setminus (\cup_{x\in \overline{D}}B(x,s))\neq\emptyset\}$ for short. 
	 If $r_i<r$, then
	\begin{align*}
	d(x,x_i)\leq \diam(D)+d(x_i,\bD)\leq\diam(D)+5r\ \hbox{ and }\ 
	B_i\subset B(x,\diam(D)+6r). 
	\end{align*}
	By (c), $r_i>s/(5+\lambda)$ for each $i\in J$. Hence, noting that $B_i,i\in I$ are disjoint and by \eqref{e:vd}
	\begin{align*}
	\#J\cdot m(B(x,\diam(D)+6r))\leq\sum_{i\in J}C_3m(B_j)\leq C_3\,m(B(x,\diam(D)+6r)),
	\end{align*}
	where $C_3=c_1(\frac{(2\diam(D)+11r)(5+\lambda)}{s})^{d_1}$. Hence, $\#J\leq C_3$. 
\end{proof}

\begin{proposition}\label{P:3.3}
Suppose that $(\sX, d, m)$ is (VD) and $D\subset \sX$
is Ahlfors regular. There is a  countable collection $\{f_i:i\in \Lambda\}\subset L^2(\sX;m)$ such that the following properties hold. 
\begin{enumerate}[\rm (D.1)]
\item $0\leq f_i\leq 1$ and $\sum\limits_{i\in \Lambda} f_i\leq \1_D$  $m$-a.e.  

\item $\diam({\rm supp}[f_i])\leq 2r_i$ for each $i\in \Lambda$. 
	 
\item $d(x_i,{\rm supp}[f_i])\leq 7r_i$ for each $i\in \Lambda$.
	
\item There is $C>1$ such that that $C^{-1}m(B_i)\leq \int_Df_i (x) m (dx)\leq C\,m(B_i)$ for each $i\in \Lambda$.
\end{enumerate}
Here, ${\rm supp}[f_i]$ denotes the support of $f_i$, that is, 
the smallest closed subset of $\sX$ such that $f=0$ $m$-a.e. on its complement. 
\end{proposition}

\begin{proof}
For $(a,b), (c,d) \in \R^2$, we say $(a,b)>(c,d)$ in dictionary order in the sense that either $a>c$ or $a=c,b>d$.
	
Recall the definition of the $\frac14$-Whitney cover  $\{B_i:=B(x_i,r_i):x_i\in \sX\setminus \overline{D},r_i>0,i\in I\}$ of $\sX\setminus \overline{D}$ from \eqref{e:3.1a}.	For each $i\in\Lambda$, by the property (D.ii), there exists $y_i\in D$ such that $d(x_i,y_i)\leq 6r_i$. By \eqref{e:ahlfors} and 
 \eqref{e:vd1}, there are constants $C_1\geq 1$ and $0<C_2\leq C_3$ such that 
\begin{align}
\label{e:2.1} m(23B_i)\leq  C_1m(B_i)  \qquad &\hbox{for every }i\in \Lambda,\\
\label{e:2.2} C_2m(B_i)\leq m\big(B (y_i,r_i) \cap D\big)\leq C_3m(B_i)  \qquad  &\hbox{for every }i\in \Lambda.
\end{align}
For each $k\in\Z$, we define
\[
\Lambda_k=\{i\in \Lambda:\, 2^{k}\leq r_i<2^{k+1}\},
\]
and we order elements in $\Lambda_k$, so we write $\Lambda_k=\{i(k,n)\}_{n=1}^{N_k}$, where $N_k$ can be $\infty$. \smallskip 

For $k\in\Z$, we let $f_i^{(k)}=0$ for $i\in \bigcup_{l<k}\Lambda_l$, and define $ f_{i}^{(k)}\in L^2(\sX;m)$ 
for $i\in \bigcup_{l=k}^\infty \Lambda_{l}$ with the following rule. 

First, as  $\frac{C_2}{2C_1C_3}\leq \frac12$, there is some 
$f_{i(k,1)}^{(k)}\in L^2(\sX;m)$ so that
\begin{equation*} 
0\leq f_{i(k,1)}^{(k)}\leq \1_{B_D(y_{i(k,1)},r_{i(k,1)})}
\end{equation*} 
and
\begin{equation*} 
\int_\sX f_{i(k,1)}^{(k)}dm=\frac{C_2}{2C_1C_3}m(B_D(y_{i(k,1)},r_{i(k,1)})).
\end{equation*}

Next, recursively, for $(l,j)>(k,1)$, there is some  $f_{i(l,j)}^{(k)}\in L^2(\sX;m)$ so that
\begin{equation}\label{e:2.3} 
0\leq f^{(k)}_{i(l,j)}\leq \1_{B_D(y_{i(l,j)},r_{i(l,j)})}\Big(1-\sum_{(n,q)< (l,j)}f_{i(n,q)}^{(k)}\Big) 
\end{equation}
and that
\begin{equation}\label{e:2.4} 
\int_\sX f_{i(l,j)}^{(k)}dm=\frac{C_2}{2C_1C_3}m\big(B_D(y_{i(l,j)},r_{i(l,j)})\big).
\end{equation}
Indeed, for
$(n,q)<(l,j)$, we notice that if
\[
\operatorname{supp}[f^{(k)}_{i(n,q)}]\cap B_D(y_{i(l,j)},r_{i(l,j)}) \neq \emptyset,
\]
then
\[B_D(y_{i(n,q)},r_{i(n,q)})\cap B_D(y_{i(l,j)},r_{i(l,j)})\neq\emptyset,\] 
which implies
\begin{align*}
d(x_{i(n,q)},x_{i(l,j)})&\leq d(x_{i(n,q)},y_{i(n,q)})+d(y_{i(n,q)},y_{i(l,j)})+d(y_{i(l,j)},x_{i(l,j)})\\
&\leq 6r_{i(n,q)}+\big(r_{i(n,q)}+r_{i(l,j)}\big)+6r_{i(l,j)}\\
&\leq 21r_{i(l,j)}.
\end{align*}
Here we have used the facts that $r_{i(n,q)}<2r_{i(l,j)}$ and $d(x_i,y_i)\leq 6r_i$ for each $i\in \Lambda$, so
\begin{equation}\label{e:2.5}
B_{i(n,q)}\subset 23B_{i(l,j)}.
\end{equation}
Hence, 
\begin{align*}
&\sum_{(n,q)<(l,j)\atop E^{(k)}_{i(n,q)}\cap B_D(y_{i(l,j)},r_{i(l,j)}) \neq \emptyset} \int_\sX f_{i(n,q)}^{(k)}dm \\
=&\sum_{(n,q)<(l,j)\atop E^{(k)}_{i(n,q)}\cap B_D(y_{i(l,j)},r_{i(l,j)}) \neq \emptyset} \frac{C_2}{2C_1C_3}m\big(B_D(y_{i(n,q)},r_{i(n,q)})\big)\\
\leq &\sum_{(n,q)<(l,j)\atop E^{(k)}_{i(n,q)}\cap B_D(y_{i(l,j)},r_{i(l,j)}) \neq \emptyset} \frac{C_2}{2C_1}m\big(B_{i(n,q)}\big)\\
\leq &\frac{C_2}{2C_1}m(23B_{i(l,j)})\\
\leq & \frac{1}{2}m\big(B_D (y_{i(l,j)},r_{i(l,j)})  \big),
\end{align*}
where we used the fact $ \int f_{i(\eta,q)}^{(k)}dm=\frac{C_2}{2C_1C_3}m\big(B_D(y_{i(\eta,q)},r_{i(\eta,q)})\big)$ for $(n,q)<(l,j)$ in the equality,   used \eqref{e:2.2} in the second inequality, we used \eqref{e:2.5} and  the fact that $B_i,i\in \Lambda$ are pairwise disjoint in the third inequality, and we used $m(23B_i)\leq C_1m(B_i)\leq \frac{C_1}{C_2}m(B_D(y_i,r_i))$ by \eqref{e:2.1} and \eqref{e:2.2} in the last inequality. It follows immediately that 
\[ 
\int_{B_D(y_{i(l,j)},r_{i(l,j)})}\Big(1-\sum_{(n,q)< (l,j)}f_{i(n,q)}^{(k)}\Big)dm\geq \frac{1}{2}m\big(B_D(y_{i(l,j)},r_{i(l,j)})\big).
\]
As $\sum_{(n,q)< (l,j)}f_{i(n,q)}^{(k)}\leq 1$ by the construction, we can find some $f_{i(l,j)}^{(k)}$ that satisfies \eqref{e:2.3} and \eqref{e:2.4}.

In summary, for each $k\in \Z$, we get a collection $\{f^{(k)}_i:i \in \Lambda\}$  
  satisfying the following three properties:
\begin{enumerate}[(D.k.1)]
\item   $\sum_{i\in \Lambda} 
f^{(k)}_i\leq \1_D$;

\item $0\leq f^{(k)}_i\leq \1_{B_D(y_i,r_i)}  \leq \1_D$ for each $i \in \Lambda$;  

\item $\int_\sX f^{(k)}_idm= \frac{C_2}{2C_1C_3}m(B_D(y_i,r_i))$ for  each $i \in \Lambda$. 
\end{enumerate}
\smallskip 

Finally, we choose a subsequence $\{k_n; n\geq 1\}$ so that $k_n\to-\infty$ as $n\to\infty$ such that $f^{(k_n)}_i$ converges weakly in $L^2(\sX;m)$ to some $f_i$ for each $i\in \Lambda$. We next verify (D.1)--(D.4). 

First, we show (D.1).  For short, we write $\<f,g\>_{L^2(\sX;m)}=\int_\sX fgdm$ for $f,g\in L^2(\sX;m)$.  
That $0\leq f_i\leq \1_{B_D(y_i,r_i)}\leq 1$ $m$-a.e. follows from the observation 
\[
0\leq\int_E f_i(x)m(dx)=\lim\limits_{n\to\infty}\<f^{(k_n)}_i,\1_E\>_{L^2(\sX;m)}\leq m(E\cap  B_D(y_i,r_i) )\hbox{ for every bounded }
E \in \sB ( \sX) ,
\]
where we used (D.k.2) in the last inequality. Next, $\sum_{i\in \Lambda}f_i\leq\1_D$ $m$-a.e. because for every bounded $E\subset \sX$,
\begin{align*}
\int_E \sum_{i\in \Lambda} f_i (x) m (dx)=\sum_{i\in \Lambda}\int_E f_i (x) m (dx)&=\sum_{i\in \Lambda}\lim\limits_{n\to\infty}\<f_i^{(k_n)},\1_E\>_{L^2(\sX;m)}\\
&\leq \lim\limits_{n\to\infty}\sum_{i\in \Lambda}\<f_i^{(k_n)},\1_E\>_{L^2(\sX;m)}\leq m(E\cap D),
\end{align*}
where  we used (D.k.1) in the last inequality.  

(D.2) and (D.3) follow from the facts that $d(x_i,y_i)\leq 6r_i$ and $0\leq f_i\leq\1_{B_D(y_i,r_i)}$ $m$-a.e. 

Finally,  by  the weak convergence of $f_i^{(k_n)}$ to $f_i$ as $n\to \infty$ and  (D.k.3), 
\begin{align*}
\int_Df_i (x) m (dx)&=\int_{B_D(y_i,r_i)}f_i (x) m (dx) =\<\1_{B_D(y_i,r_i)},f_i\>_{L^2(\sX;m)} \\
& =\lim\limits_{n\to\infty}\<1_{B_D(y_i,r_i)},f_i^{(k_n)} \>_{L^2(\sX;m)}=\frac{C_2}{2C_1C_3}m(B_D(y_i,r_i)),
\end{align*}
(D.4) then follows from \eqref{e:2.2}. 
\end{proof}

Now, we introduce the extension operator.  

In the rest of this section, we assume $(\sX,d,m,\mcE,\mcF)$ is {\rm (VD)} and satisfies {\rm Cap}$(\phi)_{\leq}$. For each $i\in I$,   let $\eta_i\in \mcF$ be such that 
\begin{equation}\label{e:3.7a}
0\leq\eta_i\leq 1,\quad \eta_i|_{\frac52B_i}=1,\quad \eta_i|_{\sX\setminus 3B_i}=0\ \hbox{ and }\ \mcE(\eta_i,\eta_i)\leq C\frac{m(B_i)}{\phi(r_i)},
\end{equation}
for some $C>0$ depending on the constants of {\rm Cap}$(\phi)_{\leq}$, \eqref{eqnphi} and \eqref{e:vd1}. 

For $i\in I$, define $\psi_i$ by 
\begin{equation} \label{e:2.6}
\psi_i(x)=
\begin{cases} \frac{\eta_i(x)}{\sum_{j\in I}\eta_j (x)} 
\quad  &\hbox{for }x\in \sX\setminus \overline{D}, \\
0 &\hbox{for }x\in   \overline{D} .
\end{cases} 
\end{equation} 
It will be shown in Lemma \ref{L:4.2} below that $\psi_i \in \sF$ for each $i\in I$. 

We record some obvious observations in the following lemma. 

\begin{lemma}\label{L:3.4} 
 Suppose that $(\sX,d,m,\mcE,\mcF)$ is {\rm (VD)} and satisfies {\rm Cap}$(\phi)_{\leq}$. We have \\
{\rm (a)} $\sum_{i\in I}\psi_i=\1_{\sX\setminus \overline{D}}$.\quad 
{\rm (b)}  $\psi_i|_{\sX\setminus3B_i}=0$.\quad 
{\rm (c)}  $0\leq\psi_i\leq 1$.
 \end{lemma}

\begin{lemma}\label{L:3.5}
Suppose that $(\sX,d,m,\mcE,\mcF)$ is {\rm (VD)} and satisfies {\rm Cap}$(\phi)_{\leq}$. Recall that $\Lambda:=\{ i\in I: 0<r_i<\diam(D)/2\}$. Then,   
\[
\sum_{i\in\Lambda}\psi_i(x)=1\ \hbox{ for each }x\in\sX\setminus\bD\hbox{ with }d(x,\bD)<\diam(D).
\] 
\end{lemma}

\begin{proof}
We fix $x\in\sX\setminus\bD$ with  $d(x,\bD)<\diam(D)$. Then, by Lemma \ref{L:3.2}(c) and Lemma \ref{L:3.4}(b), $r_i<d(x,\bD)/2<\diam(D)/2$ for each $i\in I$ such that $\psi_i(x)\neq 0$. This implies $\sum_{i\in\Lambda}\psi_i(x)=\sum_{i\in I}\psi_i(x)=1$.
\end{proof}

\medskip

\begin{definition}\label{D:3.6}
For each $u\in L^2(\overline{D};m_0)$, define a function ${\Ex} u$ on $\sX$ by 
\[
({\Ex} u)(x)=
\begin{cases}
u(x)&\hbox{ if }x\in \overline{D},\\
\sum_{i\in \Lambda } [u]_i\psi_i(x)&\hbox{ if }x\in\sX\setminus\overline{D},
\end{cases}
\] 
where 
\begin{equation}\label{e:ui}
[u]_i=\frac{\int_D f_i(x)u(x)m(dx)}{\int_D f_i(x)m(dx)}.
\end{equation}
and $\{f_i; i\in \Lambda\}$ is the countable collection of non-negative functions in Proposition \ref{P:3.3}.
\end{definition}

\medskip

\begin{proposition}\label{P:3.7}
Suppose that {\rm (VD)} and {\rm Cap}$(\phi)_{\leq}$ hold for $(\sX, d, m, \sE, \sF)$, and that $D \subset \sX$ is Ahlfors regular.
There is $C>0$ such that   for every $u\in L^2(\overline{D};m_0)$, $x_0\in \overline{D}$  and $r>0$, 
\begin{equation} \label{e:3.8a}
\int_{ B(x_0,r)}|{\Ex} u (x) |^2 m (dx)\leq C\int_{B_D(x_0,7r)}u (x)^2 m_0(dx).
\end{equation} 
In particular, ${\Ex} $ is a bounded linear extension operator from $L^2(\overline{D};m_0)$ to $L^2(\sX;m)$ 
with
\begin{equation} \label{e:3.8b}
 \int_{ \sX}|{\Ex} u (x) |^2 m (dx)\leq C\int_{D}u (x)^2 m_0 (dx).
\end{equation} 
\end{proposition}

\begin{proof}
Fix $x_0\in \overline{D}$, $r>0$ and $u\in L^2(\overline{D};m_0)$. For $i\in \Lambda$ such that $3B_i\cap B(x_0,r)\neq\emptyset$, we pick $x\in B(x_0,r)\cap 3B_i$. We have $r_i\leq \frac{1}{2}d(x,\overline{D})\leq \frac12d(x,x_0)$ by Lemma \ref{L:3.2} (c). Then, by Proposition \ref{P:3.3} (D.2) and (D.3), we see
\begin{align*}
	d(x_0,\operatorname{supp}[f_i])&\leq d(x_0,x)+d(x,x_i)+d(x_i,\operatorname{supp}[f_i])\\
	&=d(x_0,x)+3r_i+7r_i\leq 6d(x,x_0)\leq 6r,\\
	\diam(\operatorname{supp}[f_i])&\leq 2r_i\leq d(x,x_0)\leq r.
\end{align*}
Hence, 
\begin{equation}\label{e:2.7}
\{i\in \Lambda:\,3B_i\cap B(x_0,r)\neq\emptyset\}\subset \{i\in\Lambda:\, \operatorname{supp}[f_i]\subset B(x_0,7r)\}.
\end{equation}
Hence,  
\begin{align*}
\int_{B(x_0,r)\setminus \overline{D}}{\Ex u} (x)^2 m (dx)&=\int_{B(x,r)\setminus \overline{D}}\big(\sum_{i\in \Lambda}[u]_i\psi_i(x)\big)^2m(dx)\\
&\leq \int_{B(x,r)\setminus \overline{D}}\sum_{i\in\Lambda}\psi_i(x)[u]_i^2m(dx)\\
&=\sum_{i\in \Lambda:\,3B_i\cap B(x_0,r)\neq\emptyset}[u]_i^2\int_{3B_i\cap B(x_0,r)}\psi_i(x)m(dx)\\
&\leq\sum_{i\in\Lambda:\, \operatorname{supp}[f_i]\subset B(x_0,7r)}[u]^2_im(3B_i)\\
&\leq C_1\sum_{i\in\Lambda:\, \operatorname{supp}[f_i]\subset B(x_0,7r)}[u]^2_i \int_{D}f_i (x) m (dx) \\
&\leq C_1\sum_{i\in \Lambda:\, \operatorname{supp}[f_i]\subset B(x_0,7r)}\frac{\int_D u(x)^2f_i (x) m (dx)}{\int_D f_i (x) m (dx)}\int_D f_i (x) m (dx) \\
& =C_1\sum_{i\in \Lambda:\, \operatorname{supp}[f_i]\subset B(x_0,7r)}\int_D u (x)^2f_i (x) m (dx)\\
&\leq C_1\int_{B_D(x_0,7r)}u (x)^2 m (dx),
\end{align*}
where we used Cauchy–Schwarz inequality and parts (a) and (c) Lemma \ref{L:3.4} in the first inequality,
 Lemma \ref{L:3.4}(b) in the second equality, 
 \eqref{e:2.7} and parts (b) and (c) of Lemma \ref{L:3.4}in the second inequality,
   \eqref{e:vd} and Proposition \ref{P:3.3}(D.4)  in the third inequality, 
Jensen's inequality in the fourth inequality,
 and  Proposition \ref{P:3.3}(D.1) in the last inequality.  
This establishes \eqref{e:3.8a}. 
Sending $r\to \infty$ in \eqref{e:3.8a} yields the inequality \eqref{e:3.8b}.
\end{proof}

\section{Energy estimates} \label{S:4}

In this section, we establish the following energy estimate
 for the extension operator ${\Ex}$ defined in Definition \ref{D:3.6}.  
For a Dirichlet form $(\sE, \sF)$ on $L^2(\sX; m)$ and a constant $\alpha >0$, we define 
$\sE_{\alpha} (u, v):= \sE(u, v) + \alpha \int_{\sX} u(x) v(x) m(dx)$ for $u, v\in \sF$.
Similar notation applies to the active reflected Dirichlet form $(\bar \sE, \bar \sF)$ of \eqref{e:2.8}-\eqref{e:2.9}
on $L^2(\bar D; m_0)$; that is,
$\bar \sE_{\alpha} (u, v):= \bar \sE(u, v) + \alpha \int_{\overline D} u(x) v(x) m_0(dx)$ for $u, v\in \bar \sF$.

\begin{proposition}\label{P:4.1}
Suppose that $(\sX,d,m,\sE,\sF)$ satisfies {\rm (VD)}, {\rm Cap}$(\phi)_{\leq}$ and ${\bf J}_{\phi}$,
 and  $D\subset \sX$ is Ahlfors regular. 
Then $\Ex u\in\sF$ for each $u\in{\bar \sF}$. Moreover, there is a constant $C>0$ 
depending on the parameters of {\rm (VD)}, {\rm Cap}$(\phi)_{\leq}$, ${\bf J}_{\phi}$ and \eqref{eqnphi} such that  
\begin{equation} \label{e:4.1a}
\sE_1( \Ex u, \Ex u)   \leq  C \bar \sE_1 (u, u) \quad \hbox{for every } u \in \bar \sF,
\end{equation}
and
\begin{equation} \label{e:4.1b}
\int_{B(x_0,r)\times B(x_0,r)}\big({\Ex u}(x)-{\Ex u}(y)\big)^2J(dx,dy)\leq  C\int_{B_D(x_0,7r)\times B_D(x_0,14r)}\big(u(x)-u(y)\big)^2J(dx,dy)
\end{equation}
for every $u\in\bar\sF$, $x_0\in\bD$ and $0<r<\diam(D)$. 
\end{proposition}

We split the energy of the left hand side into three parts, roughly speaking, near diagonal portion in $\sX\setminus \overline{D}$, off diagonal portion in $\sX\setminus \overline{D}$ and the cross energy portion in $(\sX\setminus \overline{D})\times D$. For simplicity, we use the following notation.

\medskip

\noindent\textbf{Notation}. For each Borel set $\mathcal{O}\subset \sX\times\sX$ and a Borel measurable function $f$ defined $m$-a.e. on $\mathcal{O}$,  we define
\[
\mcE^{(\mathcal{O})}(f,f):=\int_{\mathcal{O}}\big(f(x)-f(y)\big)^2J(dx,dy).
\]

\subsection{Near diagonal energy in $\sX\setminus \overline{D}$} 
Recall  from \eqref{e:3.1a} that $\{B_i:=B(x_i,r_i):x_i\in \sX\setminus \overline{D},r_i>0,i\in I\}$ is the  $\frac14$-Whitney cover cover of $\sX\setminus \overline{D}$. In this subsection, we estimate $\mcE^{(\mathcal{O}_d)}$ locally, where 
\[
\mathcal{O}_d :=\bigcup_{i\in I}4B_i\times 4B_i. 
\]

\begin{lemma}\label{L:4.2}
Suppose that $(\sX,d,m,\sE,\sF)$ satisfies {\rm (VD)} and {\rm Cap}$(\phi)_{\leq}$. Let $\psi_i$ be the function defined by \eqref{e:2.6}.
Then, there is a positive constant $C$ depending on the parameters in {\rm (VD)} and {\rm Cap}$(\phi)_{\leq}$ such that 
$$
\psi_i\in\mcF\ \hbox{ and }\ \mcE(\psi_i,\psi_i)\leq C\frac{m(B_i)}{\phi(r_i)}\quad\hbox{ for }i\in I.
$$
\end{lemma} 

\begin{proof}
For each $f\in I$, define $g_i :=\sum_{j\in I\atop 3B_i\cap 3B_j\neq\emptyset}\eta_j\in C(\sX)$, 
where $\{\eta_j; j\in I\}$ are defined in  \eqref{e:3.7a}. 
Then, 
\begin{equation}\label{e:3.1}
	\begin{split}
		&\quad\ \int_{\sX \times  \sX}(1\vee g_i(x)-1\vee g_i(y))^2J(dx,dy)\\
		&\leq \int_{\sX \times \sX} (g_i(x)-g_i(y))^2J(dx,dy)=\mcE(g_i,g_i)\\
		&\leq \#\{j\in I:\, 3B_i\cap 3B_j\neq\emptyset\}\sum_{j\in I: 3B_i\cap 3B_j\neq\emptyset}\mcE(\eta_j,\eta_j)\leq C_1\frac{m(B_i)}{\phi(r_i)},
	\end{split}
\end{equation}
where in the last inequality, we used
 parts (a) and (b) of Lemma \ref{L:3.2}, \eqref{eqnphi}, \eqref{e:vd1} and the energy estimate of $\eta_j,j\in I$. Hence, 
\begin{align*}
	&\quad\ \int_{\sX \times \sX}(\psi_i(x)-\psi_i(y))^2J(dx,dy) \\
	&=\int_{\sX \times \sX } \Big(\frac{\eta_i(x)}{1\vee g_i(x)}-\frac{\eta_i(y)}{1\vee g_i(y)} \Big)^2J(dx,dy)\\
	&\leq2\int_{\sX \times \sX }(\eta_i(x)-\eta_i(y))^2J(dx,dy)+2\int_{\sX \times \sX}(\frac{1}{1\vee g_i(x)}-\frac{1}{1\vee g_i(y)})^2J(dx,dy)\\
	&\leq2\int_{\sX \times \sX}(\eta_i(x)-\eta_i(y))^2J(dx,dy)+2\int_{\sX \times \sX}(1\vee g_i(x)-1\vee g_i(y))^2J(dx,dy)\\
	&\leq C_2\frac{m(B_i)}{\phi(r_i)},
\end{align*}
where the equality follows from the definition of $g_i$ and $\psi_i$, the first inequality uses the fact that $0\leq\eta_i\leq 1$ and $1\vee g_i\geq 1$, the second inequality follows from the fact $1\vee g_i\geq 1$, and the last inequality follows from \eqref{e:3.1} and the energy estimate of $\eta_i$.  Finally, noting that $\psi_i\in L^2(\sX;m)$, we conclude $\psi_i\in\sF$.
\end{proof}

\begin{proposition}\label{P:4.3}
Suppose that  {\rm (VD)}, {\rm Cap}$(\phi)_{\leq}$ and ${\bf J}_{\phi,\geq}$  hold for $(\sX, d, m, \sE, \sF)$, and
 that $D\subset \sX$ is Ahlfors regular.
Then there is a constant $C>0$ such that
\[
\mcE^{\big(\mathcal{O}_d\cap(B(x_0,r)\times B(x_0,r))\big)}(\Ex u,\Ex u)\leq  C\int_{B_D(x_0,7r)\times  B_D(x_0,14r)}\big(u(x)-u(y)\big)^2J(dx,dy),
\]
for every $u\in\mcF$, $x_0\in \overline{D}$ and $r<\diam(D)$. 
\end{proposition}
\begin{proof}
Fix $x_0\in \overline{D}$ and $u\in L^2(\overline{D};m_0)$. For each $i\in \Lambda$ such that $B(x_0,r)\cap 4B_i\neq\emptyset$, we pick $x\in B(x_0,r)\cap 4B_i$. We have $r_i\leq d(x,\overline{D})\leq d(x,x_0)$ by Lemma \ref{L:3.2} (c). Then, by Proposition \ref{P:3.3} (D.2) and (D.3), we see
\begin{align*}
	d(x_0,\operatorname{supp}[f_i])&\leq d(x_0,x)+d(x,x_i)+d(x_i,\operatorname{supp}[f_i])\\
	&=d(x_0,x)+4r_i+7r_i\leq d(x_0,x)+ 11d(x,x_0)  
	=12d(x,x_0)\leq 12r,\\
	\diam(\operatorname{supp}[f_i])&\leq 2r_i\leq 2d(x,x_0)\leq 2r.
\end{align*}
Hence, 
\begin{equation}\label{14r}
\{i\in\Lambda:\,4B_i\cap B(x_0,r)\neq\emptyset\}\subset \{i\in\Lambda:\, \operatorname{supp}[f_i]\subset B(x_0,14r)\}.
\end{equation}
	
For each $i\in\Lambda$ and $r<\diam(D)$, 
\begin{eqnarray*}
&& \mcE^{\big((4B_i\times 4B_i)\cap (B(x_0,r)\times B(x_0,r))\big)}({\Ex u},{\Ex u})\\
&=&\mcE^{\big((4B_i\times 4B_i)\cap (B(x_0,r)\times B(x_0,r))\big)}({\Ex u}-[u]_i,{\Ex u}-[u]_i)\\
&=&\mcE^{\big((4B_i\times 4B_i)\cap (B(x_0,r)\times B(x_0,r))\big)} \Big( \sum_{j\in \Lambda\atop 3B_j\cap B(x_0,r)\neq\emptyset}([u]_j-[u]_i)\psi_j,\sum_{j\in \Lambda\atop 3B_j\cap B(x_0,r)\neq\emptyset}([u]_j-[u]_i)\psi_j \Big)\\
& \leq &C_1\sum_{j\in\Lambda\atop 3B_j\cap 4B_i\cap B(x_0,r)\neq\emptyset}([u]_i-[u]_j)^2\mcE(\psi_j,\psi_j) \\
&\leq &C_2\sum_{j\in\Lambda\atop 3B_j\cap 4B_i\cap B(x_0,r)\neq\emptyset}([u]_i-[u]_j)^2\frac{m(B_j)}{\phi(r_j)},
\end{eqnarray*}
where we used Lemma \ref{L:3.5} in the second equality, 
$\#\{j\in I:\,3B_j\cap 4B_i\neq \emptyset\}\leq C_1$ for some $C_1$ depending only on the parameters of (VD) by Lemma \ref{L:3.2}(b) in the first inequality,
and Lemma \ref{L:4.2} in the second inequality. Moreover, for each term in the above summation, we have 
\begin{align*}
([u]_i-[u]_j)^2\frac{m(B_j)}{\phi(r_j)}&\leq \frac{\int_{D\times D}f_i(x)f_j(y)\big(u(x)-u(y)\big)^2m(dx)m(dy)}{\int_{D\times D}f_i(x)f_j(y)m(dx)m(dy)}\cdot\frac{m(B_j)}{\phi(r_j)}\\
&\leq C_3\int_{D\times D} f_i(x)f_j(y)\frac{\big(u(x)-u(y)\big)^2}{m(B_i)\phi(r_j)}m(dx)m(dy)\\
&\leq C_4\int_{D\times D}f_i(x)f_j(y)\big(u(x)-u(y)\big)^2J(dx,dy),
\end{align*}
where we used Jensen's inequality in the first inequality,  
the fact that 
$$
m(B_i)\asymp \int_Df_i (x) m (dx), m(B_j)\asymp \int_Df_jdm
$$ 
by Proposition \ref{P:3.3} in the second inequality, 
and \eqref{e:vd1}, ${\bf J}_{\phi,\geq}$, \eqref{eqnphi} and the facts that $\frac19r_j\leq r_i\leq 9r_j$ and $d(x,y)\leq d(x,x_i)+d(x_i,x_j)+d(y,x_j)\leq 13(r_i+r_j)$ by Lemma  \ref{L:3.2}(a) and Proposition \ref{P:3.3}(D.2),(D.3) in the third inequality.   
Hence, 
\begin{align}
\begin{split}\label{eqn2.1}
&\mcE^{\big(\mathcal{O}_d\cap (B(x_0,r)\times B(x_0,r))\big)}({\Ex u},{\Ex u})\\
\leq& \, \sum_{i\in I\atop 4B_i\cap B(x_0,r)\neq\emptyset}\mcE^{\big((4B_i\times 4B_i)\cap (B(x_0,r)\times B(x_0,r))\big)}({\Ex u},{\Ex u})\\
\leq& \, C_2C_4\sum_{i\in I\atop 4B_i\cap B(x_0,r)\neq\emptyset}\sum_{j\in I\atop 3B_j\cap B(x_0,r)\neq\emptyset}
\int_{D\times D}f_i(x)f_j(y)\big(u(x)-u(y)\big)^2J(dx,dy)     \\
\leq& \,  C_2C_4\int_{B_D(x_0,14r)\times B_D(x_0,7r)}\big(u(x)-u(y)\big)^2J(dx,dy),
\end{split}
\end{align}
where we used \eqref{e:2.7}, \eqref{14r} and Proposition \ref{P:3.3} (D.i) in the last inequality. 
\end{proof}

\subsection{Off diagonal energy in $\sX\setminus \overline{D}$} Recall that $\mathcal{O}_d=\bigcup_{i\in I}4B_i\times 4B_i$. In this subsection, we consider $\mcE^{\big(\mathcal{O}_f\cap (B(x_0,r)\times B(x_0,r))\big)}$, where 
\[
\mathcal{O}_f=\big((\sX\setminus\bD)\times (\sX\setminus\bD)\big)\setminus\mathcal{O}_d.
\]

\begin{proposition}\label{P:4.4}
Suppose that  {\rm (VD)}, {\rm Cap}$(\phi)_{\leq}$ and ${\bf J}_{\phi}$  hold for $(\sX, d, m, \sE, \sF)$, 
and that $D\subset \sX$ is Alhfors regular. 
There is a constant $C>0$ so that
\[
\mcE^{\big(\mathcal{O}_f\cap (B(x_0,r)\times B(x_0,r))\big)}({\Ex u},{\Ex u})\leq  C\int_{B_D(x_0,7r)\times B_D(x_0,7r)}\big(u(x)-u(y)\big)^2J(dx,dy),
\]
for every $u\in\mcF$, $x_0\in \overline{D}$ and $0<r<\diam(D)$. 
\end{proposition}

\begin{proof}
First, we claim that 
\begin{equation}\label{e:3.4}
d(x,y)\geq \frac19\max\{r_i,r_j\}\quad\hbox{ for every }i,j\in I,\, x\in 3B_i,\,y\in 3B_j\setminus 4B_i.
\end{equation}
To show \eqref{e:3.4}, we consider two possibilities.  

Case 1: $4B_i\cap 4B_j\neq\emptyset$. In this case, $r_i\geq r_j/9$ by Lemma \ref{L:3.2} (a), and hence
\[
d(x,y)\geq d(x_i,y)-d(x_i,x)\geq 4r_i-3r_i=r_i\geq \max\{r_i,r_j\}/9.
\]  

Case 2: $4B_i\cap 4B_j=\emptyset$. In this case, we have 
\begin{align*}
&d(x,y)\geq d(\sX\setminus4B_j,y)\geq d(\sX\setminus 4B_j,x_j)-d(y,x_j)\geq 4r_j-3r_j\geq r_j,\\
&d(x,y)\geq d(x_i,y)-d(x_i,x)\geq r_i. 
\end{align*} 

\noindent Hence, \eqref{e:3.4} holds for both cases. As a consequence, 
\begin{equation}\label{e:3.5}
d(x_i,x_j)\leq 3r_i+d(x,y)+3r_j\leq 55d(x,y)\quad\hbox{ for every }i,j\in I,\, x\in 3B_i,\,y\in 3B_j\setminus 4B_i.  
\end{equation}

 We now proceed to bound
\begin{align*}
&\mcE^{\big(\mathcal{O}_f\cap (B(x_0,r)\times B(x_0,r)\big)}({\Ex u},{\Ex u})\\
=&\int_{\mathcal{O}_f\cap \big(B(x_0,r)\times B(x_0,r)\big)} \big(\sum_{i\in\Lambda}[u]_i\psi_i(x)-\sum_{j\in\Lambda}[u]_j\psi_j(y)\big)^2J(dx,dy)\\
=& \int_{\mathcal{O}_f\cap \big(B(x_0,r)\times B(x_0,r)\big)} \big(\sum_{i\in\Lambda}\sum_{j\in\Lambda}\psi_i(x)\psi_j(y)([u]_i-[u]_j)\big)^2J(dx,dy)\\
\leq &\int_{\mathcal{O}_f\cap \big(B(x_0,r)\times B(x_0,r)\big)} \sum_{i\in \Lambda}\sum_{j\in\Lambda}\psi_i(x)\psi_j(y)\big([u]_i-[u]_j\big)^2J(dx,dy)\\
\leq&\sum_{i\in\Lambda\atop 3B_i\cap B(x_0,r)\neq\emptyset}\sum_{j\in\Lambda\atop 3B_j\cap B(x_0,r)\neq\emptyset}\big([u]_i-[u]_j\big)^2\int_{3B_i\times (3B_j\setminus 4B_i)}\psi_i(x)\psi_j(y)J(dx,dy)\\
\leq&C_1\sum_{i\in\Lambda\atop 3B_i\cap B(x_0,r)\neq\emptyset}\sum_{j\in\Lambda\atop 3B_j\cap B(x_0,r)\neq\emptyset}\frac{\big([u]_i-[u]_j\big)^2m(B_i)m(B_j)}{m\big(B(x_i,d(x_i,x_j)+r_i+r_j)\big)\phi(x_i,d(x_i,x_j)+r_i+r_j)}\\
\leq&C_2\sum_{i\in\Lambda\atop 3B_i\cap B(x_0,r)\neq\emptyset}\sum_{j\in\Lambda\atop 3B_j\cap B(x_0,r)\neq\emptyset}\frac{\int_{D\times D}f_i(z)f_j(w)\big(u(z)-u(w)\big)^2m(dz)m(dw)}{m\big(B(x_i,d(x_i,x_j)+r_i+r_j\big)\phi(x_i,d(x_i,x_j)+r_i+r_j)}\\
\leq&C_3\sum_{i\in\Lambda\atop 3B_i\cap B(x_0,r)\neq\emptyset}\sum_{j\in\Lambda\atop 3B_j\cap B(x_0,r)\neq\emptyset}\int_{D\times D}f_i(z)f_j(w)\big(u(z)-u(w)\big)^2J(dz,dw)\\
\leq&C_3\int_{B_D(x_0,7r)\times B_D(x_0,7r)}\big(u(z)-u(w)\big)^2J(dz,dw).
\end{align*}
The second equality is due to Lemma \ref{L:3.5}.
The first inequality above is due to Lemma \ref{L:3.5} and the  Jensen's inequality. The second inequality is due to the 
 fact $3B_i\times 4B_i\subset 4B_i\times 4B_i\subset \mathcal{O}_d$, and hence $(3B_i\times 4B_i)\cap\mathcal{O}_f=\emptyset$.
In the third inequality, we used {\rm (VD)}, \eqref{e:vd1}, ${\bf J}_{\phi,\leq}$, \eqref{e:3.4}, \eqref{e:3.5} and \eqref{eqnphi}.
The fourth inequality holds because of the  Jensen's inequality and $\int_Df_i (x) m (dx)\asymp m(B_i),\int_Df_jdm\asymp m(B_j)$,
while the fifth inequality is due to 
 \eqref{eqnphi}, the fact $d(x,y)\leq d(x_i,y_i)+9r_i+9r_j$ for every $x\in \operatorname{supp}[f_i],y\in \operatorname{supp}[f_j]$ by Proposition \ref{P:3.3} and ${\bf J}_{\phi,\geq}$. The  last inequality holds due to the fact $\sum_{i\in\Lambda}f_i\leq 1$ and \eqref{e:2.7}. This completes the proof of the proposition.
\end{proof}

\subsection{Cross energy} Finally, we consider the energy $\mcE^{(\mathcal{O}_c\cap (B(x_0,r)\times B(x_0,r)))}$, 
where 
\[
\mathcal{O}_c :=(\sX\setminus \overline{D})\times D.
\] 

\begin{proposition}\label{P:4.5}
Suppose that  {\rm (VD)}, {\rm Cap}$(\phi)_{\leq}$ and ${\bf J}_{\phi}$  hold for $(\sX, d, m, \sE, \sF)$, and that $D\subset \sX$ is Alhfors regular. There is a constant $C>0$ so that
\[
\mcE^{\big(\mathcal{O}_c\cap (B(x_0,r)\times B(x_0,r))\big)}({\Ex u},{\Ex u})\leq  C\int_{B_D(x_0,7r)\times B_D(x_0,r)}\big(u(x)-u(y)\big)^2J(dx,dy),
\]
for every $u\in\bar\sF$, $x_0\in \overline{D}$ and $0<r<\diam(D)$. 
\end{proposition}

\begin{proof}
First, we claim that 
\begin{align}
\label{e:3.6}
d(x,y)\geq \frac{2}{5}d(x_i,y)\quad&\hbox{ for every }i\in I,\,x\in 3B_i,\,y\in D,\\
\label{e:3.7}
d(x_i,y)\geq \frac{5}{14}d(z,y)\quad&\hbox{ for every }i\in I,\, z\in \operatorname{supp}[f_i],\, y\in D.
\end{align}
To see \eqref{e:3.6}, we notice that $d(x,x_i)\leq 3r_i$ and $d(x,y)\geq d(x_i,y)-d(x,x_i)\geq d(x_i,\overline{D})-3r_i\geq 5r_i-3r_i=2r_i$ by (D.ii), so  
\[
\frac{d(x,y)}{d(x_i,y)}\geq \frac{d(x,y)}{d(x,y)+d(x_i,x)}\geq \frac{d(x,y)}{d(x,y)+3r_i}\geq \frac{2r_i}{2r_i+3r_i}=\frac25. 
\]
To see \eqref{e:3.7}, we notice that $d(x_i,y)\geq d(x_i,\bD)\geq 5r_i$ by (D.ii) and $d(x_i,z)\leq 9r_i$ by Proposition \ref{P:3.3} (D.2) and (D.3), so 
\[
\frac{d(x_i,y)}{d(z,y)}\geq \frac{d(x_i,y)}{d(x_i,y)+d(x_i,z)}\geq \frac{d(x_i,y)}{d(x_i,y)+9r_i}\geq \frac{5r_i}{5r_i+9r_i}=\frac{5}{14}.
\]
We now proceed to estimate
\begin{align*}
&\mcE^{\big(\mathcal{O}_c\cap (B(x_0,r)\times B(x_0,r))\big)}({\Ex u},{\Ex u})\\
=&\int_{(B(x_0,r)\setminus \overline{D})\times B_D(x_0,r)}\big(\sum_{i\in \Lambda}[u]_i\psi_i(x)-u(y)\big)^2J(dx,dy)\\
\leq&\int_{(B(x_0,r)\setminus \overline{D})\times B_D(x_0,r)}\sum_{i\in \Lambda}\psi_i(x)\big([u]_i-u(y)\big)^2J(dx,dy)\\
\leq&\sum_{i\in \Lambda\atop 3B_i\cap B(x_0,r)\neq\emptyset}\int_{3B_i\times B_D(x_0,r)}\big([u]_i-u(y)\big)^2J(dx,dy)\\
\leq&C_1\sum_{i\in\Lambda\atop 3B_i\cap B(x_0,r)\neq\emptyset}\int_{3B_i\times B_D(x_0,r)}\frac{\big([u]_i-u(y)\big)^2}{m\big(B(y,d(x_i,y))\big)\phi\big(d(x_i,y)\big)}m(dx)m(dy)\\
=&C_1\sum_{i\in\Lambda\atop 3B_i\cap B(x_0,r)\neq\emptyset}\int_{ B_D(x_0,r)}\frac{m(3B_i)\big([u]_i-u(y)\big)^2}{m\big(B(y,d(x_i,y))\big)\phi\big(d(x_i,y)\big)}m(dy)\\
\leq&C_2\sum_{i\in\Lambda\atop 3B_i\cap B(x_0,r)\neq\emptyset}\int_{B_D(x_0,r)\times D}\frac{f_i(z)\big(u(z)-u(y)\big)^2}{m\big(B(y,d(x_i,y))\big)\phi\big(d(x_i,y)\big)}m(dy)m(dz)\\
\leq&C_3\sum_{i\in\Lambda\atop 3B_i\cap B(x_0,r)\neq\emptyset}\int_{B_D(x_0,r)\times D}f_i(z)\big(u(z)-u(y)\big)^2J(dy,dz)\\
\leq &C_3\int_{B_D(x_0,r)\times B_D(x_0,7r)}\big(u(y)-u(z)\big)^2J(dy,dz).
\end{align*}
In the above, we used Jensen's inequality and the facts $\psi_i(x)\geq 0$, $\sum_{i\in I}\psi_i(x)=1$ in the first inequality; 
used the fact that $0\leq\psi_i\leq 1$ and $\psi_i$ supports on $3B_i$ in the second inequality; used \eqref{e:3.6}, \eqref{e:vd1}, \eqref{eqnphi}, and ${\bf J}_{\phi,\leq}$ in the third inequality; used Jensen's inequality and the fact that $\int_Dfdm\asymp m(3B_i)$ in the fourth inequality; used \eqref{e:3.7}, \eqref{e:vd1}, \eqref{eqnphi} and ${\bf J}_{\phi,\geq}$ in the fifth inequality; and we used \eqref{e:2.7} and the fact that $\sum_{i\in I}f_i\leq 1$ in the last inequality.  
\end{proof}

\medskip

\subsection{Proof of Proposition \ref{P:4.1}} 
We first recall the following lemma from \cite[Lemma 2.1]{CKW}.

\begin{lemma}\label{L:4.6}
Suppose $(\sX,d,m)$ is {\rm(VD)}. There is a constant $c>0$ that depends only on the bounds in \eqref{e:vd} and \eqref{eqnphi} so that for every $x\in \sX$ and $r>0$,
$$
\int_{\sX\setminus B(x, r)} \frac1{m(B(x, d(x,y)) \phi (d(x, y))} m(dy) \leq \frac{c}{\phi(r)}.
$$
\end{lemma}

As a consequence, we have the following estimates.

\begin{lemma}\label{L:4.7}
Suppose $(\sX,d,m,\mcE,\mcF)$ satisfies {\rm(VD)} and ${\bf J}_{\phi,\leq}$. There is a constant $c>0$ that depends only on the bounds in \eqref{e:vd}, \eqref{eqnphi} and ${\bf J}_{\phi,\leq}$ so that for every $f\in L^2(\sX;m)$ and $r>0$,
\begin{align*}
\int_{\sX\times\sX\atop d(x,y)\geq r}(f(x)-f(y))^2J(dx,dy)
\leq \frac{c}{\phi(r)}\|f\|_{L^2(\sX;m)}^2.
\end{align*}
\end{lemma}
\begin{proof}
\begin{align*}
&\quad\ \int_{\sX\times\sX\atop d(x,y)\geq r}(f(x)-f(y))^2J(dx,dy)\leq 	\int_{\sX\times\sX\atop d(x,y)\geq r}(2f(x)^2+2f(y)^2)J(dx,dy)\\
&=4\int_{\sX\times\sX\atop d(x,y)\geq r}f(x)^2J(dx,dy)= \int_{\sX}f(x)^2\big(\int_{\sX\setminus B(x,r)}J(x,y)m(dy)\big)m(dx)\\
&\leq \frac{c}{\phi(r)}\int_\sX f(x)^2m(dx)=\frac{c}{\phi(r)}\|f\|^2_{L^2(\sX;m)},
\end{align*} 
where the last inequality holds due to ${\bf J}_{\phi,\leq}$ and Lemma \ref{L:4.6}.
\end{proof}

\begin{proof}[Proof of Proposition \ref{P:4.1}]
 Estimate \eqref{e:4.1b} follows from Propositions \ref{P:4.3}, \ref{P:4.4} and \ref{P:4.5}. When $D$ is unbounded, it follows  
 by taking $r\to \infty$ in \eqref{e:4.1b}  and Proposition \ref{P:3.7} that $\Ex u\in\sF$ and \eqref{e:4.1a} holds.

It remains to prove $\Ex u\in\sF$ when $D$ is bounded. For $0<r<1$, define 
\[
D_r :=  \bigcup_{x_0\in\bD}B(x_0, r\,\diam(D)) . 
\] 
By the regularity of the Dirichlet form $(\mcE,\mcF)$ on $L^2(\sX; m)$, there is some
$\psi\in C_c(\sX)\cap \sF$ such that  
\[
\psi|_{D_{1/4}}=1,\,\ \psi|_{\sX\setminus D_{1/2}}=0\,\ \hbox{ and }\ 0\leq\psi\leq1. 
\]
First, we can check $\psi\Ex u\in\mcF$ for each $u\in\bar \sF$. In fact,
\begin{align}
&\quad\ \int_{\sX\times\sX}(\psi(x)\Ex u(x)-\psi(y)\Ex u(y))^2J(dx,dy)\nonumber\\
&\leq 2\int_{\sX\times \sX}\psi(x)^2(\Ex u(x)-\Ex u(y))^2J(dx,dy)+2\int_{\sX\times \sX}\Ex u(y)^2(\psi(x)-\psi(y))^2J(dx,dy)\nonumber\\
&\leq 2\int_{D_{1/2}\times\sX}(\Ex u(x)-\Ex u(y))^2J(dx,dy)+2\int_{\sX\times(\sX\setminus D_{1/8})}\Ex u(y)^2(\psi(x)-\psi(y))^2J(dx,dy)\nonumber\\
&\quad\ +2\int_{(\sX\setminus D_{1/4})\times D_{1/8}}\Ex u(y)^2(\psi(x)-\psi(y))^2J(dx,dy). \label{e:3.8}
\end{align}
We claim the sum is bounded by $C_0 \bar \sE_1 (u, u)$, where $C_0$ is a constant that depends only on $\psi$, $\diam (D)$ and the parameters in
 {\rm (VD)}, {\rm Cap}$(\phi)_{\leq}$ and ${\bf J}_{\phi}$. 
 Indeed, for the first term,
 by Lemma \ref{lemmavd2}, we can choose a finite subset $z_j\in \bD,1\leq j\leq N$ such that $D_{1/2}\subset\bigcup_{j=1}^N B(z_j,3\diam(D)/4)$. 
Then by \eqref{e:4.1b}, 
\[
\sum_{j=1}^N\int_{B(z_j,3\diam(D)/4)\times B(z_j,\diam(D))}(\Ex u(x)-\Ex u(y))^2J(dx,dy)\leq C_1N\bar\sE(u,u).
\]
Moreover, by Lemma \ref{L:4.7} and Proposition \ref{P:3.7},
\[
\sum_{j=1}^N\int_{B(z_j,3\diam(D)/4)\times (\sX\setminus B(z_j,\diam(D)))} (\Ex u(x)-\Ex u(y))^2J(dx,dy)\leq \frac{C_2N}{\phi(\diam(D)/4)}\|u\|^2_{L^2(\bD;m_0)}.
\]
It follows that the first term of \eqref{e:3.8} is finite. Next, we note that $\Ex u=\sum_{i\in \Lambda:\,3B_i\cap (\sX\setminus D_{1/8}) \neq\emptyset}  [u]_i\psi_i$ on $\sX\setminus D_{1/8}$, so $\Ex u$ is bounded on $\sX\setminus D_{1/8}$ as $\#\{i\in \Lambda:\,3B_i\cap (\sX\setminus D_{1/8}) \neq\emptyset \}<\infty$ by Lemma \ref{L:3.2}(e). Hence the second term of \eqref{e:3.8} is bounded by a multiple of $\mcE(\psi,\psi)$. Finally, for the third term of \eqref{e:3.8}  
\begin{align*}
&\quad\ \int_{(\sX\setminus D_{1/4})\times D_{1/8}}\Ex u(y)^2(\psi(x)-\psi(y))^2J(dx,dy)\leq
4\int_{(\sX\setminus D_{1/4})\times D_{1/8}}\Ex u(y)^2J(dx,dy)\\
&\leq 4\int_{D_{1/8}}\Ex u(y)^2\int_{\sX\setminus B(x,\diam(D)/8)}J(x,y)m(dx)m(dy)\\
&\leq 4\int_{D_{1/8}}\Ex u(y)^2\frac{C_3}{\phi(\diam(D)/8)}m(dy)\leq \frac{C_4}{\phi(\diam(D)/8)}\|u\|_{L^2(\bD;m_0)}
\end{align*}
by Lemma \ref{L:4.6} and Proposition \ref{P:3.7}. 
This proves the claim that \eqref{e:3.8} is bounded by $C_0 \bar \sE (u, u)$. Note that as $\psi\Ex u\in L^2(\sX;m)$ by \eqref{e:3.8b}, we have  $\psi\Ex u\in\sF$.

Next, observe that by Lemma \ref{L:3.2}(e),  $\#\{i\in \Lambda:\,3B_i\cap (\sX\setminus D_{1/4}) \neq\emptyset \}<\infty$.
Thus we have  by Lemma \ref{L:4.2} that 
$(1-\psi)\Ex u=\sum_{i\in \Lambda:\,3B_i\cap (\sX\setminus D_{1/4})\neq\emptyset}(1-\psi) [u]_i \psi_i\in\sF$  with 
\begin{align*}
\sE( (1-\psi)\Ex u, (1-\psi)\Ex u)^{1/2}  
&\leq \sum_{i\in \Lambda:\,3B_i\cap (\sX\setminus D_{1/4}) \neq\emptyset}   [u]_i  \left( \| 1-\psi\|_\infty \sE(\psi_i, \psi_i)^{1/2} + \|\psi_i\|_\infty \sE( \psi, \psi)^{1/2} \right) \\
&\leq  c  \sum_{i\in \Lambda:\,3B_i\cap (\sX\setminus D_{1/4}) \neq\emptyset }  \| u\|_{L^2(\bar D; m_0)}
\end{align*}
It follows that $\Ex u=\psi \Ex u+(1-\psi)\Ex u\in \sF$ and \eqref{e:4.1a} holds. 
\end{proof}

We end this section by showing that $(\bar\sE,\bar\sF)$ is a regular Dirichlet form on $L^2(\bD;m_0)$.

\begin{corollary}\label{C:4.8}
Suppose that $(\sX,d,m,\sE,\sF)$ satisfies {\rm (VD)}, {\rm Cap}$(\phi)_{\leq}$ and ${\bf J}_{\phi}$,
 and that $D\subset \sX$ is Ahlfors regular. Then, $(\bar \sE,\bar\sF)$ is a regular Dirichlet form on $L^2(\bD;m_0)$.  
\end{corollary}
\begin{proof}
First, for every $u\in C_c(\bD)$, we can find $v\in C_c(\sX)$ such that $v|_{\bD}=u$ by Tietze extension theorem. Then, by the regular property, there are $v_n\in \sF\cap C_c(\sX),n\geq 1$ such that $\|v_n-v\|_{L^\infty}\to 0$ as $n\to\infty$. Then, $u_n=v_n|_{\bD}\in{\bar \sF}\cap C_c(\bD)$ and $\|u_n-u\|_{L^\infty}\to 0$ as $n\to\infty$. This shows $C_c(\bD)\cap {\bar \sF}$ is dense in $C_c(\bD)$. 
	
Next, by Proposition \ref{P:4.1}, $\Ex u\in\sF$ for every $u\in{\bar \sF}$. There is a sequence $v_n\in\sF\cap C_c(\sX),n\geq 1$ such that $\|v_n-\Ex u\|_{\sF}\to 0$ as $n\to\infty$, where $\|f\|_{\sF}=\sqrt{\mcE(f,f)+\|f\|^2_{L^2(\sX;m)}}$ for $f\in\sF$. Then, $v_n=u_n|_\bD\in{\bar \sF}$ and $\|v_n|_{\bD}-v\|_{{\bar \sF}}\to 0$ as $n\to\infty$. This shows that $\sF\cap C_c(\bD)$ is dense in ${\bar \sF}$  with respect to the norm $\|f\|_{{\bar \sF}}=\sqrt{\bar\sE(f,f)+\|f\|^2_{L^2(\bD;m_0)}}$.
\end{proof}

\section{Heat kernel estimates}\label{S:5}

In this section, we prove that the reflected jump diffusion satisfies mixed-stable-like heat kernel estimates ${\bf HK}(\phi)$. Note that  
under ({\rm VD}) and {(\rm QRVD)}, ${\bf J}_\phi+{\rm CSJ}(\phi)\Longleftrightarrow{\bf HK}(\phi)$
by \cite[Theorem 1.13]{CKW} and \cite[Theorem 1.20]{Ma} (see Remark \ref{R:2.7}).
 In view of Lemma \ref{L:2.5}, it suffices to show  ${\rm CSJ}(\phi)$ holds for the reflected MMD $(\overline D, d, m_0, \bar \sE, \bar \sF)$. 
Below, we introduce an equivalent formulation for ${\rm CSJ}(\phi)$ that is easier to verify in our setting.

\begin{definition}\label{D:4.1}\rm 
We say that condition {\rm CSJB($\phi$)} holds for $(\sX,d,m,\mcE,\mcF)$ if there are positive constants $C_1,C_2,c$ and $\lambda\geq 3$ such that for every $r\in(0,c\,\diam(\sX))$, almost all $x_0\in \sX$ and any $f\in \mcF$, there is a cut-off function $\varphi\in\mcF$ for $B(x_0,r)\subset B(x_0,2r)$ so that 
\begin{eqnarray*}
&& \int_{B(x_0,3r)\times B(x_0,3r)}f(x)^2 \big(\varphi(x)-\varphi(y)\big)^2J(dx,dy)\\
&\leq &C_1\int_{B(x_0,\lambda r)\times B(x_0,\lambda r)}\big(f(x)-f(y)\big)^2J(dx,dy)+\frac{C_2}{\phi (r)} \int_{B(x_0,\lambda r)}f(x)^2 m(dx).
\end{eqnarray*}
\end{definition}

\medskip

\begin{lemma}\label{L:5.2}
Suppose  that {\rm (VD)} and ${\bf J}_{\phi,\leq}$ hold for $(\sX, d, m, \mcE,\mcF)$. Then {\rm CSJ$(\phi)$}$\Longleftrightarrow ${\rm CSJB($\phi$)}. 
 \end{lemma}
 
\begin{proof}
($\Longrightarrow$)  Suppose that ${\rm CSJ}(\phi)$ holds, and let $c,C_0\in(0,1],C_1,C_2$ be the parameters of ${\rm CSJ}(\phi)$. Then, for $R=r\in (0,c\,\diam(\sX))$, $x_0\in\sX$, $f\in\sF$, there is a cut off $\varphi\in\sF$ for $B(x_0,r)\subset B(x_0,2r)=B(x_0,R+r)$ such that  
\begin{equation}\label{e:4.1}
\begin{split} 
&\quad\  \int_{B(x_0,(2+C_0)r)}f^2d\Gamma(\varphi,\varphi)\\
&\leq C_1\int_{U\times U^*}(f(x)-f(y))^2J(dx,dy)+\frac{C_2}{\phi(r)}\int_{B(x_0,(2+C_0)r)}f^2dm\\
&\leq C_1\int_{B(x_0,(2+C_0)r)\times B(x_0,(2+C_0)r)}(f(x)-f(y))^2J(dx,dy)+\frac{C_2}{\phi(r)}\int_{B(x_0,(2+C_0)r)}f^2dm.
\end{split} 
\end{equation} 
Moreover, by ${\bf J}_{\phi,\leq}$ and \eqref{eqnphi}, on $\sX\setminus B(x_0,(2+C_0)r)$, 
\begin{align*}
\Gamma(\varphi,\varphi)(dx)&=\int_{\sX}(\varphi(x)-\varphi(y))^2J(dx,dy) \\
&\leq C_3m(dx)\int_{B(x_0,2r)}\frac{m(dy)}{m(B(x,d(x,y)))\phi(d(x,y))}\\
&\leq C_3m(dx)\int_{\sX\setminus B(x,C_0r)}\frac{m(dy)}{m(B(x,d(x,y)))\phi(d(x,y))}\\
&\leq C_4m(dx)/\phi(r). 
\end{align*}
Hence, 
\begin{equation}\label{e:4.2}
\int_{B(x_0,3r)\setminus B(x_0,(2+C_0)r)}f^2d\Gamma(\varphi,\varphi)\leq \frac{C_4}{\phi(r)}\int_{B(x_0,3r)\setminus B(x_0,(2+C_0)r)} f^2dm.
\end{equation}
Noting that
\[
\int_{B(x_0,3r)\times B(x_0,3r)}f(x)^2 \big(\varphi(x)-\varphi(y)\big)^2J(dx,dy)
\leq 
\int_{B(x_0,3r)}f^2d\Gamma(\varphi,\varphi). 
\]
CSJB($\phi$) with $\lambda=3$ follows from \eqref{e:4.1} and \eqref{e:4.2}. \medskip

($\Longleftarrow$)  Suppose that CSJB($\phi$) holds, and let $\lambda,c,C_1,C_2$ be the positive constants of CSJB($\phi$). We fix $x_0\in\sX$, $f\in \mcF$ and $0<r\leq R<4c\,\diam(\sX)$. We write
\[
B_l=B(x_0,R+lr/4)\ \hbox{ for }l=0,1,2,3,4.
\]
We define a cut-off $\varphi$ for $B_2\subset B_3$ as follows.

\begin{enumerate}[(i)]

\item By Lemma \ref{lemmavd2}, we can find a finite set of points $\{x_j\}_{j\in J}\subset B_2\setminus B_1$
such that $\{B(x_j,\frac{r}{8\lambda}):\,j\in J\}$ are pairwise disjoint and $\bigcup_{j\in J}B(x_j,\frac{r}{4\lambda})\supset B_2\setminus B_1$.

\item 
For each $j\in J$, by CSJB($\phi$), we choose a cut-off function $\varphi_j\in \mcF$ for 
$B(x_j,\frac{r}{4\lambda})\subset B(x_j,\frac{r}{2\lambda})$ such that 
\begin{align}
\begin{split}\label{eqn32}
&\quad\  \int_{B(x_j,\frac{3r}{4\lambda})\times B(x_j,\frac{3r}{4\lambda})} f(x)^2\big(\varphi_j (x)-\varphi_j (y)\big)^2J(dx,dy)\\
& \leq  \, C_1\int_{B(x_j,r/4)\times B(x_j,r/4)}\big(f(x)-f(y)\big)^2J(dx,dy)+\frac{C_2}{\phi (r)} \int_{B(x_j,r/4)}f(x)^2 m(dx).
\end{split}
\end{align}

\item Choose $\psi\in C_c(\sX)\cap\sF$ such that $0\leq\psi\leq1$, $\psi|_{B_1}=1$ and $\psi|_{\sX\setminus B_2}=0$. Define 
\[\varphi=(\max_{j\in J}\varphi_j)\vee \psi.\] 
\end{enumerate}
Then, $\varphi$ is a cut-off function for $B_2\subset B_3$. Automatically, $\varphi$ is a cut-off function for $B_0=B(x_0,R)\subset B_4=B(x_0,R+r)$.\medskip

Note that $\varphi=\max\limits_{j\in J}\varphi_j$ on $\sX\setminus B_1$ 
 and $ |\varphi (x)-\varphi (y)|\leq\max\limits_{j\in J}|\varphi_j (x)-\varphi_j (y)|$ for $x, y\in \sX\setminus B_1$.  We have 
\begin{eqnarray}\label{e:4.4}
 && \int_{(B_4\setminus B_1)\times (B_4\setminus B_1)\atop d(x,y)\leq \frac{r}{4\lambda}}f(x)^2\big(\varphi(x)-\varphi(y)\big)^2J(dx,dy)
 \nonumber \\
& \leq & \int_{(B_4\setminus B_1)\times (B_4\setminus B_1)\atop d(x,y)\leq \frac{r}{4\lambda}}f(x)^2\sum_{j\in J}\big(\varphi_j(x)-\varphi_j(y)\big)^2J(dx,dy)
\nonumber \\
&\leq & \sum_{j\in J}\int_{B(x_j,\frac{3r}{4\lambda})\times B(x_j,\frac{3r}{4\lambda})}f(x)^2\big(\varphi_j(x)-\varphi_j(y)\big)^2J(dx,dy) \nonumber \\
& \leq &\sum_{j\in J}\Big(C_1\int_{B(x_j,r/4)\times B(x_j,r/4)}\big(f(x)-f(y)\big)^2J(dx,dy)+\frac{C_2}{\phi (r)} \int_{B(x_j,r/4)}f(x)^2 m(dx)\Big) \nonumber \\
&\leq &C_1C_3\int_{U\times U}\big(f(x)-f(y)\big)^2J(dx,dy)+\frac{C_2C_3}{\phi (r)} \int_{U}f(x)^2 m(dx),
 \end{eqnarray}
where $U=B_4\setminus B_0$. In the second inequality, we used the fact that $\varphi_j(x)-\varphi_j(y)\neq 0$ only if $\{x,y\}\cap B(x_j,\frac{2r}{4\lambda})\neq \emptyset$ which implies $\{x,y\}\subset B(x_j,\frac{3r}{4\lambda})$. In the third inequality, we used \eqref{eqn32}. In the last inequality, we used the facts that $B(x_j,r/4)\subset B_4\setminus B_0$ for each $j\in J$ and $\sum_{j\in J}1_{B(x_j,r/4)}\leq C_3<\infty$ by Lemma \ref{lemmavd2}. 

Moreover, by   Lemmas \ref{L:4.6} and \ref{L:4.7}, \eqref{eqnphi} and the fact $\varphi$ is a cut-off function for $B_2\subset B_3$, we can easily check that 
\begin{align}
\label{e:4.5}\int_{(B_4\setminus B_1)\times (B_4\setminus B_1)\atop d(x,y)>\frac{r}{4\lambda}}f(x)^2\big(\varphi(x)-\varphi(y)\big)^2J(dx,dy)\leq \frac{C_4}{\phi (r)} \int_{B_4\setminus B_1}f(x)^2 m(dx),\\
\label{e:4.6}\int_{(B_4\setminus B_1)\times \big(\sX\setminus (B_4\setminus B_1)\big)}f(x)^2\big(\varphi(x)-\varphi(y)\big)^2J(dx,dy)\leq \frac{C_5}{\phi (r)} \int_{B_4\setminus B_1}f(x)^2 m(dx),\\
\label{e:4.7}\int_{K}f^2d\Gamma(\varphi,\varphi)\leq \frac{C_6}{\phi (r)} \int_{K}f (x) ^2 m (dx) 
\quad\hbox{ for each }K\subset \sX\setminus (B_4\setminus B_1).
\end{align}
This shows that CSJ$(\phi)$ holds with $C_0=1$ by \eqref{e:4.4}--\eqref{e:4.7}. 
\end{proof}

\begin{theorem}\label{T:5.3}
Suppose  that {\rm (VD)}, ${\bf J}_{\phi}$ and {\rm CSJ$(\phi)$} hold for $(\sX, d, m, \sE, \sF)$, and that $D\subset \sX$ is Alhfors regular. Then {\rm CSJ$(\phi)$} holds for $(\overline D, d, m_0, \bar\sE,\bar \sF)$. 
\end{theorem}

\begin{proof}
By Lemma \ref{L:5.2}, CSJB$(\phi)$ holds for $(\sX,d,m,\sE,\sF)$. We let $\lambda,c,C_1,C_2$ be the parameters of CSJB$(\phi)$ for $(\sX,d,m,\sE,\sF)$.

Fix $x_0\in \overline{D}$, $f\in\bar\sF$, $0<r<c\,\diam(D)$ and $g={\Ex} f$. By Proposition \ref{P:4.1}, $g\in \mcF$. So by CSJB$(\phi)$ for $(\sX,d,m,\sE,\sF)$, there is  a cut-off function $\varphi\in\sF$ for $B(x_0,r)\subset B(x_0,2r)$ such that  
\begin{align}\label{eqn37}
\begin{split}
&\quad\ \int_{B_D(x_0,3r)\times B_D(x_0,3r)}f(x)^2\big(\varphi(x)-\varphi(y)\big)^2J(dx,dy)\\
&\leq\int_{B(x_0,3r)\times B(x_0,3r)} g(x)^2\big(\varphi(x)-\varphi(y)\big)^2J(dx,dy)\\
&\leq C_1\int_{B(x_0,\lambda r)\times B(x_0,\lambda r)}\big(g(x)-g(y)\big)^2J(dx,dy)+\frac{C_2}{\phi (r)} \int_{B(x_0,\lambda r)}g(x)^2 m(dx).
\end{split}
\end{align}
By Propositions \ref{P:3.7} and \ref{P:4.1},   
\begin{align*}
\int_{B(x_0,\lambda r)}g (x)^2 m (dx) &\leq C_3\int_{B_D(x_0,7\lambda r)}f (x) ^2 m (dx),\\
\int_{B(x_0,\lambda r)\times B(x_0,\lambda r)}\big(g(x)-g(y)\big)^2J(dx,dy)&\leq C_4\int_{B_D(x_0,14\lambda r)\times B_D(x_0,14\lambda r)}\big(f(x)-f(y)\big)^2J(dx,dy).
\end{align*}  
It follows that CSJB$(\phi)$ holds for $(\overline D, d, m_0, \bar\sE,\bar\sF)$. 
This completes the proof of the theorem in view of Lemma \ref{L:5.2}.
\end{proof}

\medskip

\begin{proof}[Proof of Theorem \ref{mainthm}]
Theorem \ref{mainthm} is an immediate consequence of Lemma \ref{L:2.5},
Theorem \ref{T:5.3} and \eqref{e:2.6a} in Remark \ref{R:2.7}.
\end{proof}

 \hskip 0.2truein
 
\end{document}